\documentclass[12pt,reqno]{amsart}
\usepackage{amscd,amssymb}
\usepackage{graphicx,enumerate}

\usepackage{mathptmx}

\usepackage{color}
\usepackage{pgf, tikz}
\usepackage{url}
\usepackage{enumerate}

\def\MEDSCIP{\relax} 

\def\NOINDENT{\relax} 

\def\height{\operatorname{ht}}

\def\width{\operatorname{width}}

\def\conv{\operatorname{conv}}

\def\vertex{\operatorname{vert}}
\def\inte{\operatorname{int}}
\def\Hilb{\operatorname{Hilb}}
\def\vol{\operatorname{vol}}
\def\Aff{\operatorname{Aff}}
\def\join{\operatorname{join}}

\def\B{\operatorname{B}}
\def\para{\mathsf{par}}
\def\Lpara{\mathsf{Lpar}}
\def\LLL{\mathcal{L}}
\def\Lat{\mathsf{L}}

\def\b{\mathbf b}
\def\v{\mathbf{v}}

\def\e{\mathbf{e}}

\def\np{{\operatorname{\mathsf{NPol}}}}

\def\RR{\mathbb R}
\def\QQ{\mathbb Q}

\def\ZZ{\mathbb Z}
\def\NN{\mathbb N}
\def\FF{\mathbb F}

\let\phi=\varphi
\let\epsilon=\varepsilon

\newtheorem{lemma}{Lemma}[section]  \newtheorem{theorem}[lemma]{Theorem}
\newtheorem{proposition}[lemma]{Proposition} 

\theoremstyle{definition} \newtheorem{definition}[lemma]{Definition} \newtheorem{remark}[lemma]{Remark}
\newtheorem{example}[lemma]{Example} \newtheorem{question}[lemma]{Question}

\begin{document}

\title[Quantum jumps of normal polytopes]{Quantum jumps of normal polytopes}

\author{Winfried Bruns}
\author{Joseph Gubeladze}
\author{Mateusz Micha$\l$ek}

\address{Universit\"at Osnabr\"uck, Institut F\"ur Mathematics, 49069 Osnabr\"uck, Germany}

\email{wbruns@uos.de}

\address{Department of Mathematics\\
         San Francisco State University\\
         1600 Holloway Ave.\\
         San Francisco, CA 94132, USA}
\email{soso@sfsu.edu}

\address{
Freie Universit\"at\\
 Arnimallee 3\\
 14195 Berlin, Germany\newline
Polish Academy of Sciences\\
         ul. \'Sniadeckich 8\\
         00-956 Warsaw\\
         Poland}
\email{wajcha2@poczta.onet.pl}

\thanks{Supported by grants DFG BR 688/22-1 (Bruns), NSF DMS-1301487 and GNSF DI/16/5-103/12 (Gubeladze), Polish National Science Center grant no. 2012/05/D/ST1/01063 (Micha{\l}ek)}



\begin{abstract}
We introduce a partial order on the set of all normal polytopes in $\RR^d$. This poset $\np(d)$ is a natural discrete counterpart of the continuum of convex compact sets in $\RR^d$, ordered by inclusion, and exhibits a remarkably rich combinatorial structure. We derive various arithmetic bounds on elementary relations in $\np(d)$, called \emph{quantum jumps}. The existence of extremal objects in $\np(d)$ is a challenge of number theoretical flavor, leading to interesting classes of normal polytopes: minimal, maximal, spherical. Minimal elements in $\np(5)$ have played a critical role in disproving various covering conjectures for normal polytopes in the 1990s. Here we report on the first examples of maximal elements in $\np(4)$ and $\np(5)$, found by a combination of the developed theory, random generation, and extensive computer search.
\end{abstract}

\maketitle

\section{Introduction}\label{INTRO}

Normal polytopes are popular objects in combinatorial commutative algebra and toric algebraic geometry: they define the normal homogeneous monoid algebras  \cite[Ch.~2]{Kripo}, \cite[Ch.~7]{MS} and the projectively normal toric embeddings \cite[Ch.~10]{Kripo}, \cite[Ch.~2.4]{CLS}. The motivation for normal polytopes in this work is the following more basic observation: these lattice polytopes are natural discrete analogues of (continuous) convex polytopes and, more generally, convex compact sets in $\RR^d$.

Attempts to understand the normality property of lattice polytopes in more intuitive geometric or integer programming terms date back from the late 1980s and 1990s; see Section \ref{The poset}. The counterexamples in \cite{BrICP,BrGuNCov,BGHMW} to several conjectures in that direction implicitly used a certain poset $\np(d)$ of the normal polytopes in $\RR^d$.  We explicitly introduce this poset in Section \ref{The poset}. If normal polytopes or, rather, the sets of their lattice points are the discrete counterparts of convex compact sets in $\RR^d$, then the poset $\np(d)$ is the corresponding discrete analogue of the continuum of all such convex compact sets, ordered by inclusion. Put another way,  $\np(d)$ provides a formalism for the `discrete vs. continuous' dichotomy in the context of convex geometry.

In this article we focus on the discrete structure of the poset $\np(d)$, in particular the existence of maximal elements. In future studies we plan to examine the topological and finer geometric properties of the underlying order complex. One of our motivations is to study global properties of the family of normal polytopes, in analogy to moduli spaces -- not only properties of particular polytopes. The present article follows a program that was sketched in \cite{FPSAC}.

Another aim is to set up a formalism for the search of special normal polytopes (or, equivalently, projective toric varieties) by a random walk on $\np(d)$. Motivated from physics, one can consider various measures on the smallest possible changes of the polytope as analogs of potential of the jumps. Such directed search proved useful in our search for maximal polytopes in $\np(4)$ and $\np(5)$. As it turns out, random search can also hit maximal polytopes, notably in  $\np(4)$.

A pair $(P,Q)$ of normal polytopes of equal dimension is called 
 a \emph{quantum jump} if $P\subset Q$ and $Q$ has exactly one more lattice point than $P$. Here the word \emph{quantum} refers to the smallest possible discrete change of a normal polytope and also points to random walks on $\np(d)$: among all possible quantum jumps one chooses the ones according to an adopted strategy.
 
 Quantum jumps define a partial order on the set of normal polytopes in which $P<Q$ if and only if there exists an ascending chain of quantum jumps that leads from $P$ to $Q$. We consider the relation $<$ as the discrete analogue of the set theoretic inclusion between convex compact subsets of $\RR^d$.

The extent of distortion of the continuum in the suggested discretization process is encoded in extremal elements of $\np(d)$ and the topological complexity of the geometric realization of $\np(d)$: local and global properties of $\np(d)$, respectively. It is not a priori clear that $\np(d)$ exhibits any of these irregularities at all. 


Explicit nontrivial (i.e., different from unimodular simplices) minimal elements,  which were called \emph{tight} polytopes in \cite{BrGuNCov}, have been known for quite a while.  They  exist in all dimensions $\ge 4$ and special instances were crucial in disproving the unimodular covering of normal polytopes or the integral Carath\'eodory property. The existence of nontrivial minimal elements in $\np(3)$ is open.

Finding maximal elements is much more difficult but, as mentioned above, we have been successful in dimensions $4$ and $5$. There seems to be no way to construct higher dimensional maximal polytopes from a given one, but there is little doubt that maximal polytopes exist in all dimensions $\ge 4$. However, the existence in dimension $3$ remains open.

Sections \ref{Basics} and \ref{Normal polytopes} recall basic notions and results for the study of normal polytopes. 

In Section \ref{Lattice-Stratification} we study the height of a lattice point $z$ over a polytope $P$ that, roughly speaking, counts the number of lattice layers between $z$ and the facets of $P$ that are visible from $z$.
It is the natural measure for distance based on the lattice structure. We show that there is no bound on the height of quantum jumps that depends only on dimension. A very precise characterization of quantum jumps in dimension $3$ is obtained at the end of Section \ref{Lattice-Stratification}.

Section \ref{Bounding-jumps} contains a sharp bound for quantum jumps in all dimensions. It is roughly proportional to the product of dimension and the lattice diameter of $P$ that we call \emph{width}. It shows that there are only finitely many jumps $(P,Q)$ for fixed $P$ and allows us to find them efficiently.

Section \ref{Spherical-states} is devoted to special normal polytopes defined by spheres and, more generally, ellipsoids. In particular we prove that all quantum jumps are infinitesimally close to the initial polytope relative to the size of the latter when the shape approximates a sphere with sufficient precision. The question on normality of the convex hulls of all lattice points in ellipsoids naturally arises. In dimension $3$ we always have the normality, and our experiments did not lead to counterexamples in dimensions 4 and 5.

The final Section \ref{Maximal states} describes our experimental approach to the existence of maximal polytopes. The main difficulty was to find a criterion that lets us choose terminating ascending chains with some positive probability. The maximal polytopes were eventually found by a combination of random generation and directed search. The computational power of Normaliz \cite{Nmz} has been proved invaluable for these experiments.

\bigskip\noindent\emph{Acknowledgment.} We thank B. van Fraassen for his comments in the early stages of this project. We are grateful to anonymous reviewers for their helpful comments and spotting several inaccuracies.

\section{Basic notions}\label{Basics}

The sets of nonnegative integer and real numbers are denoted, respectively, by $\ZZ_+$ and $\RR_+$. The Euclidean norm of a vector $v\in\RR^d$ is $\|v\|$. We write $\e_1,\ldots,\e_d$ for the standard basis vectors of $\RR^d$. A \emph{point configuration} is a finite subset $X\subset\ZZ^d$.
For a subset $X\subset\RR^d$ we set $\Lat(X)=X\cap\ZZ^d$.

For a more detailed account and the proofs for the statements in this section we refer the reader to \cite[Ch. 1 and 2]{Kripo}.

\subsection{Polytopes}\label{Polytopes}
An \emph{affine subspace} of a Euclidean space is a shifted linear subspace. An affine map between two affine spaces is a map that respects \emph{barycentric coordinates}. Equivalently, an affine map is the restriction of a linear map between the ambient vector spaces followed by a parallel translation.

Two subsets $X,Y\subset\RR^d$ are called \emph{unimodularly equivalent} if there is an \emph{integral-affine isomorphism} $X\to Y$, i.e., there is an affine isomorphism $f:\RR^d\to\RR^d$, mapping $\ZZ^d$ bijectively to itself, such that $f(X)=Y$.

A \emph{closed affine half-space} $H^+\subset\RR^d$ is a subset of the form
$$
\{x\in\RR^d\ :\ h(x)\ge0\}\subset\RR^d,
$$
where $h:\RR^d\to\RR$ is a non-zero affine map and $H=h^{-1}(0)$ is the bounding hyperplane. When $h$ is a linear map, the half-space $H^+$ and the hyperplane $H$ will be called \emph{linear} or \emph{homogeneous}.

An affine subspace $A$ is \emph{rational} if it is spanned by points in $\QQ^d$. The bounding hyperplane $H$ of a \emph{rational} half-space $H^+$ is given in the form
\begin{equation}\label{half-form}
H=\{x\in\RR^d: \alpha_1x_1+\dots+\alpha_dx_d+\beta\ge 0\}
\end{equation}
with an integer $\beta$ and coprime integers $\alpha_1,\dots,\alpha_d\in\ZZ^d$.

The \emph{affine hull} $\Aff(X)$ of a subset $X\subset\RR^d$ is the smallest affine subspace of $\RR^d$ containing $X$. The \emph{convex hull} of $X$ will be denoted by $\conv(X)$.

All considered polytopes are assumed to be \emph{convex}, i.e., a polytope is the convex hull of a finite subset $X\subset \RR^d$. Equivalently, a polytope $P$ is a bounded intersection of finitely many closed affine half-spaces: $P=\bigcap_{i=1}^n H_i^+$. The \emph{faces} of $P$ are the intersections of the form $P\cap H$ where $H^+\subset\RR^d$ is an affine closed affine half-space containing $P$. Also $P$ is a face of itself. The \emph{vertices} of $P$ are the $0$-dimensional faces of $P$, and the $(d-1)$-dimensional faces of $P$ are called the \emph{facets}. The vertex set of $P$ will be denoted by $\vertex(P)$. A \emph{simplex} is a polytope whose number of vertices exceeds the dimension of the polytope by one.

A full-dimensional polytope $P\subset\RR^d$ admits a unique representation $P=\bigcap_{i=1}^n H_i^+$, where the $H_i^+\subset\RR^d$ are closed affine half-spaces and $\dim(P\cap H_i)=d-1$, $i=1,\ldots,n$. We call this representation the \emph{irreducible representation} of $P$.

For every polytope $P\subset\RR^d$, its interior and the boundary with respect to $\Aff(P)$ will be denoted, respectively, by $\inte(P)$ and $\partial P=P\setminus\inte(P)$.

A \emph{lattice polytope} $P\subset\RR^d$ is a polytope whose vertices are \emph{lattice points}, i.e., elements of $\ZZ^d$. A \emph{rational polytope} has its vertices in $\QQ^d$.

A lattice simplex is called \emph{unimodular} if the edge vectors at some (equivalently, any) vertex define a part of a basis of $\ZZ^d$.

Let $X$ be a subset of $\RR^d$ such that $\Aff(X)=\Aff(\Lat(\Aff(X)))$, for example a lattice polytope. Then we can assign to $X$ a \emph{normalized volume}: it is the measure on $\Aff(X)$ in which a unimodular simplex in $\Aff(X)$ has volume $1$. If $\Aff(X)=\RR^d$, the normalized volume equals $d!$ times the Euclidean $d$-volume and will be denoted by $\vol_d$. Note that the normalized volume is invariant under integral-affine transformations, but not under Euclidean isometries of $\RR^d$ if $\dim\Aff(X)<d$.

For a full-dimensional lattice polytope $P\subset\RR^d$, the normalized volume $\vol_d(P)$ is a natural number. Moreover, $P$ is a unimodular simplex if and only if $\vol_d(P)=1$.

Let $P\subset\RR^d$ and $Q\subset\RR^e$ be two lattice polytopes. A $(d+e+1)$-dimensional lattice polytope $R$ is a \emph{join} of $P$ and $Q$ if it is unimodularly equivalent to the \emph{standard join}, defined by
\begin{align*}
\join(P,Q)=\conv\biggl((P,\underbrace{0,\ldots,0}_{e+1}),(\underbrace{0,\ldots,0}_d,1,Q)\biggr)\subset\RR^{d+e+1}.
\end{align*}
Thus unimodular simplices are joins of lattice points.

A $d$-dimensional polytope will be called a $d$-polytope. 
For a lattice $d$-polytope $P\subset\RR^d$ and a facet $F\subset P$, there exists a unique affine map, the \emph{facet-height function,} $\height_F:\RR^d\to\RR$ such that $\height_F(\ZZ^d)=\ZZ$, $\height_F(F)=0$, and $\height_F(P)\subset\RR_+$. If $\Aff(F)$ bounds the half-space $H^+$, then with the notation introduced for \eqref{half-form},
$$
\height_F(x)=h(x)=\alpha_1x_1+\dots+\alpha_dx_d+\beta.
$$

A lattice polytope $P\subset\RR^d$ is a \emph{unimodular pyramid} over $Q$ if $P$ is a join of the polytope $Q$ and a lattice point. The polytope $Q$ serves as the base and the additional point serves as the apex of the pyramid $P$. If $\dim P=d$, then $P$ is a unimodular pyramid over $Q$ if $Q$ is a facet of $P$ and $\Lat(P)\setminus\Lat(Q)$ is a single point that has height $1$ over $Q$.

\subsection{Cones and Hilbert bases}\label{Cones}
A \emph{conical set} $C\subset\RR^d$ is a subset of $\RR^d$ for which $\lambda x+\mu y\in C$ whenever $x,y\in C$
and $\lambda,\mu\in\RR_+$. A \emph{cone} means a \emph{finitely generated, rational,} and \emph{pointed conical set}. That is, a cone $C\subset\RR^d$ is a subset such that $C=\RR_+x_1+\cdots+\RR_+x_n$ for some $x_1,\ldots,x_n\in\ZZ^d$ and there is no nonzero element $x\in C$ with $-x\in C$.

For a cone $C$, the additive submonoid $\Lat(C)\subset\ZZ^d$ has a unique minimal generating set, which is the set of \emph{indecomposable elements} of the (additive) submonoid $\Lat(C)\subset\ZZ^d$. This set is called the \emph{Hilbert basis} of the cone $C$ and denoted by $\Hilb(C)$.

A cone $C$ is called \emph{unimodular} if $\Hilb(C)$ is a part of a basis of $\ZZ^d$. Equivalently, $C$ is unimodular if $\Hilb(C)$ is a linearly independent set.

The \emph{faces} of a cone $C\subset\RR^d$ are the intersections of type $H\cap C$ for a homogeneous half-space $H^+\subset\RR^d$, containing $C$. Also $C$ is a face of $C$.  Among the faces of $C$ we have the \emph{extremal rays} and \emph{facets}.

A non-zero lattice vector $x\in\ZZ^d$, $x\neq0$, is called \emph{primitive} if it is the generator of the monoid $\Lat(\RR_+v)\cong\ZZ_+$. This holds if and only if the coordinates of $x$ are coprime.  The primitive lattice vectors in the extremal rays of a cone $C$ are called the \emph{extremal generators} of $C$. They form a part of $\Hilb(C)$.

Assume $d>0$. For a facet $F$ of a $d$-cone $C\subset\RR^d$ there is a unique linear map, the \emph{face-height function,} $\height_F:\RR^d\to\RR$ such that $\height_F(F)=0$, $\height_F(\ZZ^d)=\ZZ$, and $\height_F(C)=\RR_+$. The last two equalities are equivalent to the condition that $\height_F(\Lat(C))=\ZZ_+$.

Every lattice polytope $P\subset \RR^d$ defines the cone $C(P)\subset\RR^{d+1}$ as follows. One embeds $P$ into $\RR^{d+1}$ by identifying $x\in P$ with the point $(x,1)\in P\times\{1\}$ and chooses $C(P)$ as the union of the rays originating from $0$ and passing through a point of $P\times\{1\}$. Then $C(P)$ is generated by the vectors $(x,1)$, $x\in \vertex(P)$, and these vectors are the extremal generators of $C(P)$.

We call the process of attaching the $(d+1)$st coordinate $1$ the \emph{homogenization of coordinates}. In order to facilitate it, we set $x'=(x,1)$ for $x\in\RR^d$. For $y\in\RR^{d+1}$ the $(d+1)$st coordinate is called its \emph{degree}.

\subsection{Simplicial cones and simplices}\label{Simplicial}
A cone of dimension $d$ is \emph{simplicial} if it has exactly $d$ extremal rays, or, equivalently, its extremal generators are linearly independent.

Let $v_1,\dots,v_d\in \ZZ^d$ be linearly independent and let $C=\RR_+v_1+\dots+\RR_+v_d$ be the simplicial cone spanned by them. Moreover, let $U=\ZZ v_1+\dots+\ZZ v_d$ be the sublattice  and $M=\ZZ_+v_1+\dots+\ZZ_+v_d$ be the affine monoid generated by  $v_1,\dots,v_d$. The group $U$ acts on $\RR^d$ by translations, and a fundamental domain of this action is
$$
\para(v_1,\dots,v_d)=\{a_1v_1+\dots+a_dv_d: 0\le a_i<1,\ i=1,\dots,d\}.
$$
The set
$$
\Lpara(v_1,\dots,v_d)=\para(v_1,\dots,v_d)\cap \ZZ^d
$$
of the lattice points represents the orbits of $U$ in $\ZZ^d$, or, in other words, the residue classes of $\ZZ^d$ modulo $U$. Based on these observations one easily proves \cite[Prop. 2.43]{Kripo}:

\begin{proposition}\label{SimpMod}
With the notation just introduced, the following hold:
\begin{enumerate}[{\rm(a)}]
\item $E=\Lpara(v_1,\dots,v_d)$ is a system of generators of the $M$-module $\Lat(C)$ (in the self-explanatory terminology);
\item $(x+M)\cap(y+M)=\emptyset$ for $x,y\in E$, $x\neq y$;
\item $\#E=\bigl[\ZZ^d:U\bigr]$;
\item $\Hilb(C)\subset \{v_1,\dots,v_d\}\cup E$.
\end{enumerate}
\end{proposition}

Since we are usually interested in the cone $C$ spanned by $v_1,\dots,v_d$, we can and will assume that $v_1,\dots,v_d$ are the extremal generators of $C$. The semi-open parallelotope $\para(v_1,\dots,v_d)$ is called the \emph{basic parallelotope} of $C$, and we call
$$
\mu(C)=\bigl[\ZZ^d:U\bigr]
$$
the \emph{multiplicity} of $C$. One has $\mu(C)=\vol(S)$ where $S$ is the \emph{basic simplex} with vertices $0,v_1,\dots,v_d$.

Let $F$ be a facet of $C$ and let $v$ be the extremal generator opposite to $F$. Then we have
$$
\mu(C)=\height_F(v)\mu(F);
$$
see \cite[Prop. 3.9]{Kripo}. This formula reflects a stratification of $\Lpara(v_1,\dots,v_d)$:

\begin{proposition}\label{ht-points}
$\Lpara(v_1,\dots,v_d)$ contains exactly $\mu(F)$ lattice points of height $j$ over $F$, $j=0,\dots,\height_F(v)-1$.
\end{proposition}

\begin{proof}
Let $m=\height_F(v)$. The group homomorphism
$$
p:\ZZ^d\to\ZZ/m\ZZ,\quad p(x)=\height_F(x) \pmod {m\ZZ},
$$
factors through $\ZZ^d/U$. Thus each class in $\ZZ^d/U$ decomposes into $\mu(C)/m$ classes that have the same height over $F$ modulo $m$.
For each $j=0,\dots,m-1$ we must have $\mu(F)$ such classes.
\end{proof}

Let $\Delta$ be a lattice $d$-simplex. Then $C(\Delta)$ is a simplicial cone, and we may write $\para(\Delta)$ for the basic parallelotope of $C(\Delta)$ and $\Lpara(\Delta)$ for its lattice points. Note that
$$
\mu(C(\Delta))=\vol_d(\Delta).
$$
The nonzero points in $\Lpara(\Delta)$ are stratified into layers of constant degree. Clearly the maximum degree in $\Lpara(\Delta)$ is at most $d$ and the minimum nonzero degree is at least $1$.  There seems to be no complete description of this stratification, but in special cases one has more information.

We are particularly interested in the case in which $\Delta$ is \emph{empty}: the only lattice points of $\Delta$ are its vertices. In this case one can say a little more:

\begin{lemma}\label{empty}
Suppose $\Delta$ is an empty simplex. Then $\Lpara(\Delta)$ has no points in degrees $1$ and $d$. In particular, if $\dim \Delta=3$, then the nonzero elements of $\Lpara(\Delta)$ live in degree $2$.
\end{lemma}

\begin{proof} That there are no points in degree $1$ is the definition of `empty', and that there are no points of degree $d$ follows if one applies the point reflection $\rho:\RR^d\to\RR^d$ at the midpoint of $\para(\Delta)$
$$
\rho(x)=(v_1'+\dots+v_{d+1}')-x,
$$
where $v_1,\dots,v_{d+1}$ are the vertices of $\Delta$.
\end{proof}

\section{Normal polytopes}\label{Normal polytopes}

In this section we introduce the class of normal polytopes, recall basic facts and several explicit families. We also define an order structure on the set of normal polytopes in $\RR^d$.

\subsection{Normal polytopes}\label{Facts}
\begin{definition}\label{normal-polytope}
A polytope $P\subset\RR^d$ is \emph{normal} if, for all $c\in\NN$, one has
$$
\Lat(cP)=\{x_1+\cdots+x_c\ :\ x_1,\ldots,x_c\in\Lat(P)\}.
$$
\end{definition}

The continuous version of Definition \ref{normal-polytope} is the equality
$$
cX=\underbrace{X+\cdots+X}_c,\qquad c\in\NN,
$$
satisfied for any convex subset $X\subset\RR^d$, where the left hand side is the $c$-th dilation and the right hand side the Minkowski sum of $c$ copies of $X$.

This version of `normal' is used in many sources. It is called `integrally closed' in \cite[Ch. 2]{Kripo} where `normal' was used for a weaker property, namely the normality of the monoid $M(P)$ defined below. Further, in \cite[p.~4]{CHHH} the authors distinguish the \emph{Integral Decomposition Property} (IDP) from normality. The first one is referring to the ambient lattice, as we do, the second one to the lattice generated by integral points of the polytope.  
In this paper we prefer the more succinct `normal' to the more algebraically oriented `integrally closed'.

By Pick's 19th century theorem \cite[Cor. 2.54]{Kripo} all lattice polygons are normal. But in high dimensions, starting with 3, the normal polytopes form a small portion of all lattice polytopes.

The next theorem encapsulates some basic facts about normal polytopes, the parts (a), (b), (c) and (d) being direct consequences of Definition \ref{normal-polytope} and the parts (e), (f), (g), (h) and (i) being proved in \cite[2.81]{Kripo}, \cite[2.57]{Kripo}, \cite[3.1]{BrGuNCov}, \cite{Pay} and \cite{GuLong}, respectively.

\begin{theorem}\label{Encapsulates}
\begin{enumerate}[{\rm (a)}]
\item
A lattice polytope that is unimodularly equivalent to a normal polytope, is normal.
\item
If $P$ is a union of normal polytopes then $P$ is normal.
\item
If $P$ is normal then every face of $P$ is normal.
\item
Cartesian products and joins of normal polytopes are normal. Unimodular pyramids over normal polytopes are normal; in particular, unimodular simplices are normal.
\item
If $P$ is normal then for every complete flag of faces
$$
\FF:\quad F_0\subset F_1\subset\ldots\subset F_{d-1}\subset P,\qquad d=\dim P,
$$
there exists an $\FF$-incident unimodular $d$-dimensional simplex $\Delta\subset P$, i.e.,
$$
\dim(\Delta\cap F_i)=i,\qquad i=1,\dots,d-1.
$$
\item For any lattice polytope $P$, the dilated polytopes $cP$ are normal as soon as $c\ge\dim P-1$.
\item Lattice parallelotopes (not necessarily rectangular) of any dimension are normal.
\item A full dimensional lattice polytope $P\subset\RR^d$ is normal if the primitive normal vectors to the facets of $P$ form a subset of a root system, whose irreducible summands are of type $A$, $B$, $C$, or $D$.
\item A lattice polytope $P$ is normal if its every edge contains at least $4d(d+1)+1$ lattice points. When $P$ is a lattice simplex this bound can be lowered to $d(d+1)+1$.
\end{enumerate}
\end{theorem}

Call a lattice polytope $P\subset\RR^d$ \emph{smooth} if the primitive edge vectors at every vertex $v\in P$ form a part of a basis of $\ZZ^d$. The terminology is explained by the observation that the projective toric variety of a lattice polytope is smooth if and only if $P$ is smooth; see \cite[p. 371]{Kripo}. \emph{Oda's question} asks whether smooth lattice polytopes are normal; see, for instance, \cite{ZZ} and the references therein.

The recent extensive treatment \cite{HPPS} of unimodular triangulatons for various classes of normal polytopes presents the state of the art in the field.

For every lattice polytope $R$ the set $\Lat(R)'=\{x'=(x,1):x\in \Lat(P)\}$ generates an affine submonoid $M(R)$ of $\ZZ^{d+1}$, and the normality of $R$ (as used in this paper) is equivalent to $\Hilb(C(R))=\Lat(R)'$, or, in other words, to the equality
$$
\sum_{x\in\Lat(R)}\ZZ_+(x,1)=\Lat(C(R))\qquad (\subset\ZZ^{d+1}),
$$
of graded affine monoids, where the degree is chosen as introduced in Section \ref{Cones}.  We set $\overline M(R)=\Lat(C(R))$. In algebraic terms, it is the integral closure of $M(R)$ in $\ZZ^{d+1}$ (and in general larger than the normalization of $M(R)$, the integral closure in the sublattice generated by $M(R)$). See \cite[Ch. 2]{Kripo} for an extensive discussion.  The $k$-th degree components of the monoids just introduced will be denoted by $M(R)_k$ and $\overline M(R)_k$, respectively.

As a consequence of Lemma \ref{empty} one can show (see \cite[Th. 2.52]{Kripo}):

\begin{lemma}\label{HilbDeg}
Let $P$ be a lattice polytope. Then all elements of $\Hilb(P)$ have degree $\le d-1$.
\end{lemma}

\subsection{The poset $\np(d)$}\label{The poset}

\begin{definition}\label{npol(d)}
The partially ordered set $\np(d)$ is the set of normal polytopes in $\RR^d$, ordered as follows: $P<Q$ if and only if there exists a finite sequence of normal polytopes of the form

\begin{equation}\label{shrinking} \begin{aligned} &P=P_0\subset\ldots\subset
P_{n-1}\subset P_n=Q,\\ &\#\Lat(P_i)=\#\Lat(P_{i-1})+1,\qquad i=1,\ldots,n. \end{aligned}
\end{equation}
\end{definition}

One easily observes that, in the sequence above, if $\dim(P_i)=\dim(P_{i-1})+1$ then $P_i$ is a unimodular pyramid over $P_{i-1}$.

The importance of the poset $\np(d)$ is explained as follows. In the late 1980s, in an attempt to give a more succinct characterization of the normal point configurations, the following two distinguished conjectures were proposed in \cite{Sebo}, of which the second one had been already asked as a question in \cite{CFS}:

\smallskip\noindent\emph{Unimodular Cover (UC):} A lattice polytope $P$ is normal if and only if $P$ is the union of unimodular simplices.

\smallskip\noindent\emph{Integral Carath\'eodory Property (ICP):} A lattice polytope $P\subset\RR^d$ is normal if and only if for an arbitrary natural number $c\in\NN$ and an arbitrary integer point $z\in\Lat(cP)$ there exist integer points $x_1,\ldots,x_{d+1}\in\Lat(P)$ and integer numbers $a_1,\ldots,a_{d+1}\in\ZZ_+$ with $z=a_1x_1+\cdots+a_{d+1}x_{d+1}$ and $a_1+\cdots+a_{d+1}=c$.

\smallskip Informally, (UC) says that the continuity, modeled by normal polytopes, is piece-wise by nature, resulting from the constituent unimodular simplices. The (ICP) is an arithmetic version of (UC). (The original conjectures were formulated for general cones $C$ that are not necessarily of the form $C(P)$ for normal $P$, using $\Hilb(C)$ instead of $\Lat(P)$.)

The first indication of the relevance of the poset $\np(d)$ was the following observation in \cite{BrGuNCov}, without introducing the poset structure in $\np(d)$ explicitly: for both conjectures (UC) and (ICP) it is critical to check their validity on the \emph{minimal elements} of the poset $\np(d)$. The mentioned polytopes are called \emph{tight polytopes} in \cite{BrGuNCov}. The very existence of tight polytopes is not quite intuitive:
computational evidence shows that many descending sequences of the type (\ref{shrinking}) reversed lead to complete erasure of the initial normal polytope. However, tight polytopes have popped up in dimensions 4 and higher, and the larger the dimension the more frequently so. A counterexample to (UC) was finally found in \cite{BrGuNCov}. In \cite{BGHMW} it was shown that the example also disproves (ICP). That (ICP) is strictly weaker than (UC) was shown in \cite{BrICP}, where the same strategy of shrinking normal polytope was used with the following important refinement: if a shrinking process halts at a minimal
counterexample to (UC) then chances are that, somewhere along the descent path, the stronger property (UC) is lost before (ICP). This can be viewed as the second indication of the relevance of the poset $\np(d)$ in understanding the normality property.

Every normal 3-dimensional polytope that is comparable with a unimodular simplex within $\np(3)$, is covered by unimodular simplices, as follows immediately from \cite[Lemma 2.2]{BrGuNCov}. In particular, the lack of nontrivial minimal elements, if true, would imply that (UC) holds for normal polytopes of dimension $3$. (UC) is open in dimension $4$ as well, but there are nontrivial minimal elements in $\np(4)$ so that the same argument cannot work.

One easily generates infinitely many higher dimensional minimal normal polytopes from a single one. In fact, for any minimal element $P\in\np(d)$ and any element $Q\in\np(e)$ the product polytope $P\times Q$ is a minimal element of $\np(d+e)$. However, as this paper shows, the situation is very different for maximal normal polytopes -- so far we have been able to find only a handful (up to unimodular equivalence) maximal normal polytopes in $\np(4)$ and $\np(5)$.

\section{Lattice stratifications and quantum jumps}\label{Lattice-Stratification}

In this section we single out the elementary relations in the poset $\np(d)$ between two full dimensional polytopes as the main object of our study and show that, already in dimension 3, their arithmetic picture is quite involved.

\subsection{Large empty layers around polytopes}\label{Height-over-polytope} In the following it will be very convenient to say that a facet $F$ of a $d$-polytope $P\subset\RR^d$ is \emph{visible} from $x\in\RR^d\setminus P$ if for every $y\in F$ the line segment $[x,y]$ intersects $P$ exactly in $y$. Note that $\height_F(x)<0$ if and only if $F$ is visible from $x$ because the points in $P$ have nonnegative height by convention.

\begin{definition}\label{height-over-polytope}
Let $P\subset\RR^d$ be a full-dimensional lattice polytope.
\begin{enumerate}[{\rm (a)}]
\item For a point $z\in\ZZ^d\setminus P$ the \emph{height $\height_P(z)$ of $z$ over $P$} is defined by
$$
\height_P(z)=\max\bigl(-\height_F(z)\ :\ F\subset P\ \text{a facet, visible from}\ z\bigr).
$$
The points in $P$ have height $0$ over $P$.

\item For $j\in\ZZ_+$, the polytope $P^{-j}$ is defined by
$$
P^{-j}=\{x: \height_P(x)\le j\}.
$$
\item The \emph{lattice stratum of height} $j$ around $P$ is the subset
$\partial P^{-j}\cap\ZZ^d$.
\item The \emph{width} of $P$ with respect to a facet $F\subset P$ and the \emph{absolute width} are defined as follows
\begin{align*}
&\width_FP=\max\bigl(\height_F(x)\ :\ x\in P\bigr),\\
&\width P=\max\bigl(\width_FP\ :\ F\subset P\ \text{{a facet}}\bigr).\\
\end{align*}
\end{enumerate}
\end{definition}

The term `stratum' above is justified: one has the stratification
\begin{equation}\label{lattice-stratification}
\ZZ^d\setminus P=\bigcup_{j=1}^\infty\big(\Lat(\partial
P^{-j}).
\end{equation}
Informally, $\Lat(\partial P^{-j})$ consists of lattice points outside of the polytope $P$ on `lattice distance' $j$ from $P$. The
polytopes $P^{-j}$ are rational polytopes, but usually \emph{not} lattice polytopes. In fact, as we will see below, $\Lat(\partial
P^{-j})$ can very well be empty.
\begin{remark}
If $P$ is a normal polytope then it defines a normal projective toric variety $X$ together with a very ample line bundle $\LLL$ providing a projectively normal embedding. The points $\Lat(P)$ correspond to a basis of global sections $H^0(X,\LLL)$ \cite[Ch.~10.B]{Kripo}\cite[Ch.~4.3]{CLS}. Let $K$ be the canonical (Weil) divisor on $X$. Points of $\Lat(P^{-j})$ correspond to a basis of global sections $H^0(X,\LLL-jK)$. In particular, if some strata are empty this is equivalent to the fact that by adding the (effective) divisor $-K$ we do not obtain any new global sections.

Many of our results may be interpreted in this language. For example, Theorem \ref{infinitestrata} below implies that there is no lower bound on $j$, even for normal toric threefolds $X$ with a very ample line bundle $\LLL$, guaranteeing the inequality $h^0(X,\LLL)<h^0(X,\LLL-jK)$.
\end{remark}

Two easy observations:
\begin{enumerate}[{\rm (i)}]
\item
If a point $z\in\ZZ^d\setminus P$ is at the smallest possible positive height above $P$, then
$$
\Lat(\conv(P,z))=\Lat(P)\cup\{z\}.
$$
\item
$\width_F P$ is always attained at a vertex of $P$ not in $F$. It is positive since $P\neq F$ .
\end{enumerate}

Without any constraints on a lattice point $z$ and a lattice polytope $P$, except the requirement $\Lat(\conv(P,z))=\Lat(P)\cup\{z\}$, there is no upper bound for the heights $\height_P(z)$, not even for normal $3$-polytopes $P$. The simplest such example is the unit tetrahedron $P=\conv(0,\e_1,\e_2,-\e_3)$ and the points $z_k=\e_1+\e_2+k\e_3$: we have $\height_P(z_k)=k$. Although one should note that none of the lattice strata around a unimodular simplex of any dimension is empty.

Our next result shows that there is no dimensionally uniform upper bound for the height of the \emph{lowest} lattice points above lattice polytopes, not even in the class of normal polytopes, and not even in dimension 3.

\begin{figure}[hbt]
\begin{center}
\tikzset{facet style/.style={opacity=1.0,very thick,line,join=round}}
\begin{tikzpicture}[y  = {(-0.4cm,-0.6cm)},
                    z  = {(0.9659cm,-0.25882cm)},
                    x  = {(0cm,0.8cm)},
                    scale = 2]
\draw [->,dashed] (-1.5, 0, 0) -- (2,0,0) node at (2.3,0,0) {$x_1$};

\draw [->,dashed] (0, -2.5, 0) -- (0,2.5,0) node at (0,2.8,0) {$x_2$};

\draw [->,dashed] (0, 0, -3.5) -- (0,0,3.5) node at (0,0,3.7){$x_3$};

\draw[thick] (1,0,0) -- (0,2,0) -- (0,0,3) -- cycle;
\draw[very thin] (-1,0,0) -- (0,2,0) -- (0,0,3) -- cycle;
\draw[thick] (1,0,0) -- (0,-2,0) -- (0,0,3) -- cycle;
\draw[very thin] (-1,0,0) -- (0,-2,0) -- (0,0,3) -- cycle;
\draw[thick] (1,0,0) -- (0,2,0) -- (0,0,-3) -- cycle;
\draw[very thin] (-1,0,0) -- (0,2,0) -- (0,0,-3) -- cycle;
\draw[thick] (1,0,0) -- (0,-2,0) -- (0,0,-3) -- cycle;
\draw[very thin] (-1,0,0) -- (0,-2,0) -- (0,0,-3) -- cycle;
\end{tikzpicture}
\caption{The polytope of Theorem \ref{infinitestrata}}
\end{center}
\end{figure}
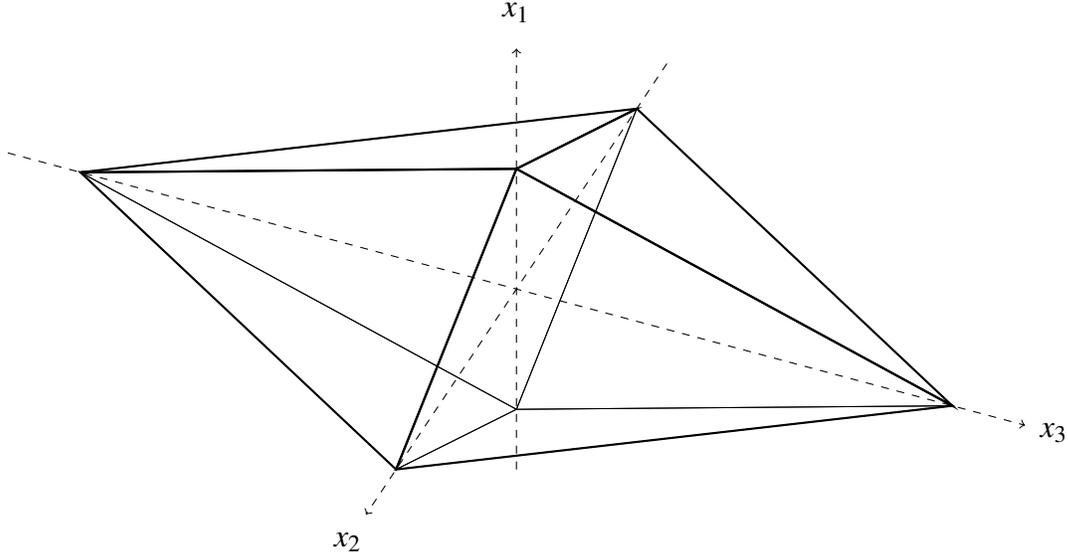

\begin{theorem}\label{infinitestrata}
There is a sequence of normal 3-polytopes $P_k\subset\RR^3$ and lattice points $z_k\in\ZZ^d\setminus P_k$, $k\in\NN$, such that for every index $k$ we have:
\begin{enumerate}[{\rm (a)}]
\item
$\width P_k=2k(k+1)(k^2+k+1)$,
\item
the lattice strata around $P_k$ up to height $k-1$ ($\approx\sqrt[4]{\frac12\width P_k}$ as $k\to\infty$) are all empty.
\item $\height_P(z_k)=k$ and $\conv(P_k,z_k)\in\np(d)$.
\end{enumerate}
\end{theorem}

As mentioned in  observation (i) above, in part (c) we also have $\Lat(\conv(P_k,z_k)\setminus P_k)=\{z_k\}$.

\begin{proof}
Consider the cross-polytopes
\begin{align*}
P_k=\conv\big(\pm k\e_1, \pm(k+1)\e_2,\pm(k^2+k+1)\e_3\big)&\subset\RR^3,\qquad k\in\NN.
\end{align*}
Each $P_k$ is the union of eight congruent copies of  the rectangular tetrahedron
$$
\Delta_k=\conv\big(0,k\e_1, (k+1)\e_2,(k^2+k+1)\e_3\big)\subset\RR^3.
$$
For every $k$ the tetrahedron
$$
\conv\big(0,k\e_1, (k+1)\e_2,\e_3\big),\quad k\in\NN,
$$
is a unimodular pyramid over the right triangle $\conv\big(0,k\e_1, (k+1)\e_2\big)$ and so is normal. Therefore, by \cite[Th. 1.6]{BrGuRect}, the tetrahedra
\begin{align*}
\Delta_k=\conv\big(0,k\e_1, (k+1)\e_2,(k(k+1)+1)\e_3\big)
\end{align*}
are normal for all $k\in\NN$. By Theorem \ref{Encapsulates}(a), the cross-polytopes $P_k$ are normal for all $k\in\NN$.

To complete the proof of (b), because of reasons of symmetry between the coordinate orthants in $\RR^3$, it is enough to show that
\begin{equation}\label{empty-positive-strata}
\min\big(-\height_{F_k}(z)\ :\ z\in\ZZ_+^d\setminus\Delta_k\big)\ge k,\quad k\in\NN,
\end{equation}
where $F_k$ is the facet of $\Delta_k$ opposite to $0$. The corresponding height function $\height_{F_k}:\ZZ^3\to\ZZ$ is given by
\begin{align*}
(\xi_1,\xi_2,\xi_3)\mapsto-(k+1)(k^2+k+1)&\xi_1-k(k^2+k+1)\xi_2-\\
-k(k+1)&\xi_3+k(k+1)(k^2+k+1).\\
\end{align*}

The lattice point $z_k=(k-1,1,1)$ belongs to $\Delta_k$ and satisfies $\height_{F_k}(z_k)=1$.
Because $(1,-1,-1)=k\e_1-z_k$ and $k\e_1\in\vertex(\Delta_k)$, the parallel translates
$$
(j,-j,-j)+F_k\subset\RR^3,\quad j\in\NN,
$$
of the triangle $F_k$ live in the planes that are defined correspondingly by $\height_{F_k}(-)=-j$. Since these are unimodular triangles, we have
\begin{multline*}
\big((j,-j,-j)+\Aff(F_k)\big)\cap\ZZ^3=\\
(k+j,-j,-j)+\ZZ(-k,k+1,0)+\ZZ(-k,0,k^2+k+1).
\end{multline*}
Therefore, for every natural number $k$, the inequality (\ref{empty-positive-strata}) is equivalent to the system of equalities
\begin{equation}\label{outside}
\begin{aligned}
\big((k+j,-j,-j)+\ZZ(-k,k+1,0)+\ZZ(-k,0,k^2+k+1)\big)&\cap\RR^3_+=\emptyset,\\
&j=1,\ldots,k-1.\\
\end{aligned}
\end{equation}

For the mentioned range of $j$, if the components of the triple
$$
(k+j,-j,-j)+a(-k,k+1,0)+b(-k,0,k^2+k+1)
$$
are positive for some $a,b\in\ZZ$ then $-j<0$ implies $a,b>0$ and $k+j\le 2k-1$ implies $a+b<2$. This contradiction proves (\ref{outside}) and, hence, (b).

\MEDSCIP Because we know the height functions of the facets, (a) follows easily.

\MEDSCIP To prove (c), we put $z_k=(0,1,k^2+1)$. One can immediately check that $z_k$ has height $k$ and that only the facets in the orthants of $\RR^3$ with sign patterns $+++$ and $-++$ are visible from $z_k$. By symmetry it is enough to prove that $\conv(\Delta_k,z_k)$ is normal.

According to Theorem \ref{Dim3} below, it is enough to exhibit lattice points in $\Delta_k$ that have heights $1,\dots,k-1$ over $F$: $y_j=(k-j,j,j)$, $j=1,...,k-1$, is in $\Delta_k$ and has height $j$ over $F$.
\end{proof}

While there is no upper bound on the number of strata around $P$ that do not contain a lattice point, we have the following uniform bound depending only on $\width P$.

\begin{proposition}\label{height-width-bound}
For all natural numbers $d$ and every lattice polytope $P\subset\RR^d$ there is a lattice point $z\notin P$ such that $\height_P(z)\le \width P$. In particular, there is a point $z\notin P$ such that $\Lat(\conv(P,z))=\Lat(P)\cup \{z\}$ and $\height_P(z)\le\width P$.
\end{proposition}

\begin{proof}
Choose an edge of $P$ and consider its two endpoints $u,v$. Then $z=u+2(v-u)\notin P$ since it lies on the straight line through $u$ and $v$ and does not belong to the edge.

For any face $F$ of $P$ one has
$$
\height_F(z)=\height_F(2v-u)=2\height_F(v)-\height_F(u)\ge -\height_F(u)\ge -\width_F P.
$$
(Note that the coefficients in $2v-u$ sum to $1$.)
\end{proof}

\subsection{Quantum jumps}\label{Quantum-jumps}
\begin{definition}\label{jump-height}
\begin{enumerate}[{\rm (a)}]
\item
A \emph{minimal (resp. maximal) polytope} is a minimal (resp. maximal) element of $\np(d)$.
\item
A pair of $d$-polytopes $(P,Q)$ in $\np(d)$ with $P<Q$ and $\#\Lat(Q)=\#\Lat(P)+1$ will be called a \emph{quantum jump from $P$}, or simply a \emph{jump (of dimension $d$)}. If $z$ is the additional lattice point in $Q$, we will say that $z$ is a \emph{quantum jump} over $P$.
\item
The \emph{height} of a jump $(P,Q)$ is defined to be the height  over $P$ of the only lattice point in $Q\setminus P$; we denote this number by $\height(P,Q)$.
\end{enumerate}
\end{definition}

\begin{remark}\label{measures}
There are several natural measures one can associate to a jump $(P,Q)$, of which the height is one. Examples include the \emph{volume} $\v(P,Q)$, equal to $\vol_d(Q\setminus P)$, and the \emph{base} $\b(P,Q)$, equal to the sum of $(d-1)$-volumes of the facets $F\subset P$, visible from the vertex of $Q$ outside $P$, normalized correspondingly with respect to the lattices $\ZZ^d\cap\Aff(F)$. Both these measures are natural numbers. If $z=\vertex(Q)\setminus P$ then the equality
$$
\v(P,Q)=\b(P,Q)\height(P,Q)
$$
is equivalent to the condition that $z$ is on same height with respect to any facet $F\subset P$, visible from $z$.

The chains in $\np(d)$, consisting of jumps that maximize the volumes at each step, lead to normal polytopes in which the lattice points are relatively rarefied. One may think that such ascending chains have potential to lead to maximal elements in $\np(d)$: after all, the lattice points in a normal polytope are meant to be regularly distributed. The reality is not as simple though; see Section \ref{Maximal states}.
\end{remark}

\begin{example}\label{Dark vertices}\emph{(Dark vertices of polygons)} The order in $\np(2)$ coincides with the inclusion order on the lattice polygons in $\RR^d$ and, consequently, the order complex of $\np(2)$ is topologically trivial, i.e., contractible. In fact, all lattice polygons are normal (Theorem \ref{Encapsulates}(f)) and if $P\subsetneq Q$ in $\np(2)$ and $v$ is a vertex of $Q$, not in $P$, then $(Q_1,Q)$ is a jump, where $Q_1=\conv(\Lat(Q\setminus\{v\})$. Iterating the process, we find a finite descending sequence 
\begin{align*}
Q=Q_0\supset&Q_1\supset\ldots\supset Q_n=P,\\
&(Q_{i+1},Q_i)\ \text{a jump for every}\ i,\\
&\qquad\qquad\qquad n=\#\Lat(Q)-\#\Lat(P).
\end{align*}  

Although no polygon can be maximal, constructing jumps from a given polygon is not quite straightforward. Let us say that a vertex $v$ of $P$ is \emph{dark} if there is no jump $z$ over $P$ such that $v$ is visible (or `illuminated') from $z$. The origin is a dark vertex of the polygon with vertices
$$
(0,0),\ (0,1),\ (1,0),\ (5,1),\ (1,5).
$$
In fact, every jump $z$ from which $(0,0)$ is visible must have one coordinate equal to $-1$. But each of these points has height $>1$ over one of the other facets. See Figure~\ref{DarkVert}; the dashed lines are the lines of height $-1$ over the facets parallel to them.

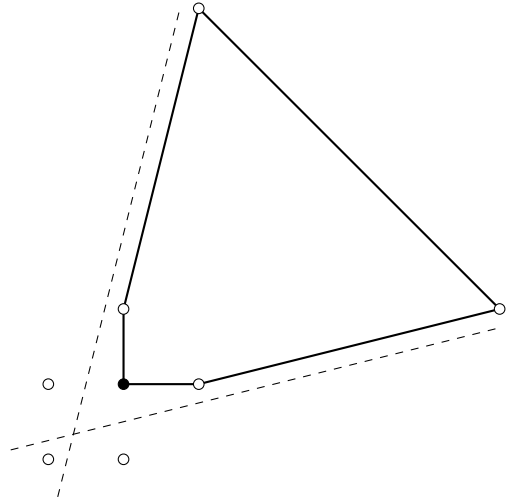
\begin{figure}[hbt]
\begin{center}
\tikzset{facet style/.style={opacity=1.0,very thick,line,join=round}}
\begin{tikzpicture}
\draw[thick] (0,1) -- (0,0) -- (1,0 ) --(5,1) -- (1,5) -- cycle ;
\filldraw[fill=black] (0,0) circle (2pt);
\filldraw[fill=white] (1,0) circle (2pt);
\filldraw[fill=white] (0,1) circle (2pt);
\filldraw[fill=white] (5,1) circle (2pt);
\filldraw[fill=white] (1,5) circle (2pt);
\filldraw[fill=white] (0,-1) circle (2pt);
\filldraw[fill=white] (-1,-1) circle (2pt);
\filldraw[fill=white] (-1,0) circle (2pt);
\draw[dashed] (-1.5,-0.875) -- (5,0.75);
\draw[dashed] (-0.875,-1.5) -- (0.75,5);
\end{tikzpicture}
\caption{A polygon with a dark vertex}
\label{DarkVert}
\end{center}
\end{figure}

If we add $(20,5)$ as a further vertex, then $(0,0)$ and $(1,0)$ will become dark. This construction can be continued an arbitrary number of steps: if $v_{-2},v_{-1},v_0,\dots,v_{n+2}$ have been constructed such that $v_0,\dots,v_{n}$ are dark, choose the next vertex $v_{n+3}$ at height $5$ over $[v_n,v_{n+1}]$  and height $1$ over $[v_{n+1},v_{n+2}]$. This will lead to a polygon with an arbitrary number of adjacent dark vertices at which the corner cones are unimodular. By the standard technique of toric desingularization \cite[Ch.~11]{CLS}, we can change the polygon to a smooth one keeping the adjacent dark vertices untouched and still dark.

If the construction is continued infinitely many times, it yields an unbounded polygon $P$  with all dark vertices and unimodular corner cones. Equivalently, all lattice points outside $P$ have infinite height over it.
\end{example}

\begin{example}\label{ex:3-order}\emph{(The poset $\np(3)$)}
As shown above, the order in $\np(2)$ is simply the inclusion order. Although it is open whether extreme elements apart from unimodular simplices exist in $\np(3)$, the example below, found by computer search, shows that the inclusion order is finer than the one induced by jumps.

Consider the 3-polytope $P$ with vertices:
$$(0,0,2),(0,0,1),(0,1,3),(1,0,0),(2,1,2),(1,2,1).$$
It is a normal lattice polytope with two additional lattice points:
$(1,1,2),(1,1,1).$

Removing either the first \emph{or} the second vertex and taking the convex hull of the other lattice points in $P$ yields a nonnormal polytope. However, if $Q$ is the convex hull of all lattice points in $P$ apart from the first \emph{and} the second vertex, then $Q$ is a normal polytope. Clearly $Q$ is inside $P$, but $Q\not<P$. 
More examples similar to the one above can be found.  
\end{example}

\begin{example}\label{Unimodulr simplices}\emph{(Unimodular simplices)}
Any two unimodular $d$-simplices in $\RR^d$ belong to same connected component of $\np(d)$. In fact, let $\Delta_1$ and $\Delta_2\subset\RR^d$ be two unimodular simplices and $v\in\Delta_1$ and $w\in\Delta_2$ be vertices. Choose a lattice broken line $[v_1,v_2,\ldots,v_k]$ in $\RR^d$, where $v=v_1$, $w=v_k$, and $v_{j+1}-v_j\in\ZZ^d$ is a primitive vector for every $j=1,\ldots,k-1$. Then we have $v<\Delta_1$, $w<\Delta_2$, and $v_j,v_{j+1}<[v_j,v_{j+1}]$.

If, in addition, $\dim\Delta_1=\dim\Delta_2=d$, then the two simplices can be even connected by quantum jumps. To this end, we first reduce the general case to the case when $\Delta_1$ and $\Delta_2$ share a vertex. Let $[v_1,\ldots,v_k]$ be a broken line as above. There exist unimodular $d$-simlices $T_1,\ldots,T_{k-1}$ such that $v_j,v_{j+1}\in\vertex(T_j)$ for every $j=1,\ldots,k-1$. In particular, it is enough to connect by quantum jumps the simplices in each of the doublets
$$
\{\Delta_1,T_1\},\ \{T_1,T_2\},\ \ldots,\ \{T_{k-1},\Delta_2\}.
$$
But the simplices in each of these pairs share a vertex. At this point without loss of generality we can assume that $0$ is a vertex of $\Delta_1$ and $\Delta_2$. Let $x_1,\ldots,x_d\in\Delta_1$ and $y_1,\ldots,y_d\in\Delta_2$ be the other vertices. Consider the two matrices in $GL_d(\ZZ)$:
$A=[x_1\ldots x_d]$ and $B=[y_1\ldots y_d]$.
By an appropriate enumeration of the nonzero vertices, we can further assume $\det A=\det B=1$. Then, because $\ZZ$ is a Euclidean domain, every integer matrix with determinant 1 is a product of elementary matrices: $SL_d(\ZZ)=E_d(\ZZ)$. Equivalently, we can transform $\{x_1,\ldots,x_d\}$ into $\{y_1,\ldots,y_d\}$ by a series of successive elementary transformations of the following two types
\begin{align*}
&\{z_1,\ldots,z_p,\ldots,z_q,\ldots,z_d\} \longrightarrow\{z_1\,\ldots,z_p,\ldots,z_q+z_p,\ldots,z_d\},\\
&\{z_1,\ldots,z_p,\ldots,z_q,\ldots,z_d\} \longrightarrow\{z_1\,\ldots,z_p,\ldots,z_q-z_p,\ldots,z_d\}.\\
\end{align*}
In particular, it is enough to show that, for a basis $\{z_1\ldots,z_d\}\subset\ZZ^d$  and two natural numbers $1\le p\not= q\le d$, the unimodular simplices in each of the pairs
\begin{align*}
\{\conv(0,z_1,\ldots,z_d),\ \conv(0,z_1\,\ldots,z_p,\ldots,z_q+z_p,\ldots,z_d)\},\\
\{\conv(0,z_1,\ldots,z_d),\ \conv(0,z_1\,\ldots,z_p,\ldots,z_q-z_p,\ldots,z_d)\}
\end{align*}
can be connected by quantum jumps. For simplicity of notation we can assume $p=1$ and $q=2$. Now the desired jumps are provided by:
\begin{align*}
\conv(0,z_1,z_2,\ldots,z_d)<\conv(0,z_1,z_1+z_2,z_2,\ldots,z_d)>\conv(0,z_1,z_1+z_2,\ldots,z_d),\\
\conv(0,z_1,z_2,\ldots,z_d)<\conv(0,z_1,z_1-z_2,z_2,\ldots,z_d)>\conv(0,z_1,z_1-z_2,\ldots,z_d),
\end{align*}
where the middle polytopes are normal, each being the union of two unimodular simplices:
\begin{align*}
\conv(0,z_1,z_1+z_2,z_2,\ldots,z_d)=\conv(0,z_1,z_2,\ldots,z_d)\cup\conv(z_1,z_1+z_2,z_2\ldots,z_d),\\
\conv(0,z_1,z_1-z_2,z_2,\ldots,z_d)=\conv(0,z_1,z_2,\ldots,z_d)\cup\conv(0,z_2-z_1,z_2,\ldots,z_d).
\end{align*}
\end{example}

Below, in Theorems \ref{Dim3}, \ref{ParaCrit}, and \ref{norm-criterion}, we will give useful criteria for a pair of lattice polytopes to be a quantum jump. In dimension $2$ the situation is very simple.

\begin{proposition}\label{non-maximal}
Let $P$ be a normal polytope.
\begin{enumerate}[{\rm (a)}]
\item If $z$ is a height $1$ lattice point over $P$, it is a quantum jump. In particular, the first lattice stratum around any maximal polytope is empty.
\item If $\dim P\le2$ then every quantum jump over $P$ has height $1$.
\end{enumerate}
\end{proposition}

\begin{proof}
(a) Clearly there are no additional lattice points in $Q=\conv(P,z)$. Let $F$ be a facet of $P$ that is visible from $z$. Then $F$ is normal by Theorem \ref{Encapsulates}(c), and $\conv(F,z)$, being a unimodular pyramid over $F$, is normal by Theorem \ref{Encapsulates}(d). Thus $Q$ is normal by Theorem \ref{Encapsulates}(b).

\MEDSCIP\NOINDENT(b) This is obvious in dimension $1$. In dimension $2$, let $F$ be a facet of $P$ that is visible from $z$, and let $\Delta$ be a unimodular line segment in $F$. Then $\conv(\Delta,x)$ is an empty triangle and therefore unimodular. But this implies $\height_F(z)=-1$.

\end{proof}

\subsection{Quantum jumps in dimension 3}\label{In-dimension-3}
In dimension $3$ we have a rather detailed description of quantum jumps $(P,\conv(P,z))$, which uses the subdivision of $P$ according to the rays emerging from $z$ and passing through a facet $F$ visible from $z$: we set
$$
P_{z,F}=\{x\in P: [x,z]\cap F\neq\emptyset\};
$$
see Figure \ref{PzF}.

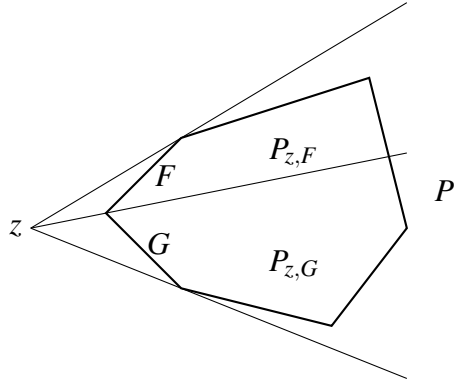
\begin{figure}[hbt]
\begin{center}
\tikzset{facet style/.style={opacity=1.0,very thick,line,join=round}}
\begin{tikzpicture}
\draw[thin] (0,0) -- (5,3) node at (-0.2,0){$z$};
\draw[thin] (0,0) -- (5,-2);
\draw[thin] (0,0) -- (5,1);
\draw[thick] (2,1.2) -- (1,0.2) -- (2,-0.8) --( 4,-1.3) -- (5,0) -- (4.5, 2) -- cycle node at (5.5,0.5){$P$} node at (3.5,1){$P_{z,F}$} node at (3.5,-0.5){$P_{z,G}$} node at (1.8,0.7){$F$} node at (1.7,-0.2){$G$};
\end{tikzpicture}
\caption{The subdivision of $P$ by facets visible from $z$}
\label{PzF}
\end{center}
\end{figure}

\begin{theorem}\label{Dim3}
Let $P\subset Q$ be lattice $3$-polytopes such that $P$ is normal and $\#\Lat(Q)=\#\Lat(P)+1$. Let $z$ be the additional lattice point in $Q$. Then the following are equivalent:
\begin{enumerate}[{\rm(a)}]
\item $z$ is a quantum jump over $P$.
\item For each facet $F$ of $P$ that is visible from $z$, the polytope $P_{z,F}$ contains at least (equivalently: exactly) $\mu(F)$ lattice points $y$ such that $\height_F(y)=j$, $j=1,\dots,\height_F(z)-1$.
\end{enumerate}
\end{theorem}

\begin{proof}
Let $F$ be a facet of $P$ that is visible from $z$. Since $F$ has dimension $2$ it has a unimodular triangulation $\Sigma$. Let $\Delta$ be a triangle in $\Sigma$. For (b) $\implies$ (a) it is enough to show that all degree $2$ lattice points in $C(Q)\setminus C(P)$ are reducible (Lemma \ref{HilbDeg}).

Let $y$ be such a point. Since the tetrahedra $\conv(\Delta,z)$ form a triangulation of $\conv(F,z)$ there are two possibilities for $y$:
\begin{enumerate}[{\rm(i)}]
\item $y$ is in the boundary of one (or more) cones $C(\Delta,z)$;
\item $y$ is in the interior of exactly one such cone.
\end{enumerate}
In case (i) $y$ is reducible since the facets of $\conv(\Delta,z)$ are unimodular: they are empty triangles. Thus the Hilbert basis of such a facet lives in degree $1$, and the degree $1$ lattice points different from $z'$ are of the form $x'$ with $x\in P$.

In case (ii) we have $y\in\Lpara(\Delta,z)$. Observe that there is exactly one point in $\Lpara(\Delta,z)$ that has height $j$, $j=1,\dots,m-1$, over the facet $\Delta\subset\conv(\Delta,z)$, where $m=\height_F(z)$ (Proposition \ref{ht-points}). We want to show that $y=u'+z'$ for a lattice point $u$ of height $m-j$ in $P_{z,F}$.

It is enough to show that
\begin{equation}\label{dim-3-jump}
\bigcup_{\Sigma}\Lpara(\Delta,z)=\{w'+z'\ :\ w\ \text{a lattice point of height}\ m-j\ \text{in}\ P_{z,F}\}.
\end{equation}
Equality (\ref{dim-3-jump}) follows if we show that every point $w'+z'$, where $w'$ is as in (\ref{dim-3-jump}), is in the interior of $\para(\Delta,z)$ for one of the triangles $\Delta$.

Clearly $w'+z'$ is outside $C(P)$ since it has negative height over $F$ (considered as a facet of $P$). But it is in $C(Q)$. So one of the alternatives (i) or (ii) applies. Since the lattice points satisfying (i) are sums $v'+z'$ with $v$ a vertex of $\Delta$ for some $\Delta$ and $w\notin F$, the alternative (i) is excluded. So (ii) applies, and we get indeed the desired $\mu(F)$ points --- one in each $\Lpara(\Delta,z)$.

Similarly, for (a)$\implies$(b), for each $\Delta$ we consider the point of height $j$  and degree $2$ in $\Lpara(\Delta,z)$. As $Q$ is normal, each such point must be a sum of two homogenized points in $\Lat(Q)$, one of which has to be equal to $z'$. All the other points must be different, belong to $P_{z,F}$, and have height $m-j$ over $F$.
\end{proof}

\begin{remark}\label{emptyjump}
Theorem \ref{Dim3} can of course be used to analyze the jumps over specific polytopes. For example, let $P=2\Delta_{pq}$ where $\Delta_{pq}$ is the empty $3$-simplex spanned by $0, \e_1,\e_3,q\e_1+p\e_2+\e_3$, $1\le q\le p-1$, $p,q$ coprime. Then $P$ has facets of multiplicity $4$, but for each facet $F$ only a single lattice point of height $1$ over $F$. Thus a quantum jump over $P$ must have height $1$ (and such exists). It is an old result of White \cite{White} that all empty 3-simplices are unimodularly equivalent to the $\Delta_{pq}$.
\end{remark}

\begin{remark}\label{Jump3}
All normal $3$-polytopes $P$ that have been encountered in our experiments, millions of them have the following remarkable property:  every point in the lowest non\-empty stratum over $P$ is a jump. On the other hand, the jumps in $\np(3)$ need not be confined to the lowest nonempty stratum.

This changes completely in dimension $4$. There exist $4$-polytopes over which there is no jump at all (see Section \ref{Maximal states}), but there are examples where jumps exist and none of them belongs to the lowest nonempty stratum.
\end{remark}

Despite all the information on $\np(3)$ at hand, we do not know whether there are maximal normal 3-polytopes (or nontrivial minimal elements). In one special case we can provide the answer.

\begin{proposition}
There are no simplices that are maximal elements of $\np(3)$.
\end{proposition}

\begin{proof}
Let $S\in\np(3)$ be a simplex with vertices $v_0,v_1,v_2,v_3$ and $F=\conv(v_1,v_2,v_3)$. Among the lattice points not in $S$, from which the only visible facet is $F$, let $z$ minimize $|\height_F(z)|$. Say $\height_F(z)=-k$. We claim that $z$ is a quantum jump. We have $\Lat(\conv(S,z))=\Lat(S)\cup\{z\}$. Because $F$ is the only visible facet, we have $S_{z,F}=S$, using notation as in Theorem \ref{Dim3}. Hence, by the mentioned theorem, we only have to check that $S$ contains $\mu(F)$ lattice points of height $j$ over $F$ for $j=1,\dots,k-1$.

Note that there are no lattice points in $S$ with height $\height_F(v_0)-1,\dots,\height_F(v_0)-k+1$. Indeed, if such a point existed, we could consider the ray starting at $v_0$ passing through that point. The first point on that ray outside $S$ would contradict the choice of $z$. In particular $\height_F(v_0)\geq k$.
Moreover, in view of the inversion map
$$
\Lpara(S)\to\ZZ^4,\quad m\rightarrow v_0'+v_1'+v_2'+v_3'-m,
$$
this implies that there are no points of degree three and of height $j$ over $F$ for $j=1,\dots,k-1$ in $\Lpara(S)$. Hence, all height $j$ points must appear in degree two and one. However, if such a point $q$ of degree two existed, then the only facet visible from the point
$$
w=v_1'+v_2'+v_3'-q=(v_1'+v_2'+v_3'+v_0'-q)-v_0'
$$
would be  $F$, and $\height_F(w)=-j$, which would contradict the choice of $z$. We thus conclude that the height $j$ points must appear in degree one. But there are exactly $\mu(F)$ such points by Proposition \ref{ht-points}.
\end{proof}
As it turns out, already in dimension 4 there are maximal normal simplices, see~Section \ref{Maximal states}.
\section{Bounding quantum jumps in $\np(d)$}\label{Bounding-jumps}

In this section we derive a bound for the heights of quantum jumps in all dimensions and show that this bound is sharp.

We begin with a criterion for a quantum jump.

\begin{theorem}\label{ParaCrit}
Let $P\subset Q$ be lattice $d$-polytopes such that $P$ is normal. Suppose that  $\Lat(Q)=\Lat(P)\cup\{z\}$. 
For every facet $F$ of $P$ that is visible from $z$, let $\Sigma_F$ be a triangulation of $F$. Then the following are equivalent:
\begin{enumerate}[{\rm(a)}]
\item $Q$ is normal.
\item For each facet $F$ of $P$ that is visible from $z$ and every $(d-1)$-simplex $\Delta\in\Sigma_F$ one has
$$
y-z'\in C(P)
$$
for every $y\in \Lpara(\Delta,z)$ with $\height_F(y)<0$.
\end{enumerate}
\end{theorem}

\begin{proof}
Suppose that $Q$ is normal and let $y$ be one of the points as in (b). Then $y\in C(Q)\setminus C(P)$. On the other hand, the Hilbert basis of $C(Q)$ is given by the vectors $x'$, $x\in\Lat(P)$, and $z'$. So $z'$ must appear in a representation of $y$ as a sum of Hilbert basis elements. The ray from $z'$ towards $y$ leaves the simplicial cone $C(\Delta,z)$ through the facet $C(\Delta)$ and thus passes through $C(F)$. Since $y-z'\in C(Q)$ lies on this ray and has positive height over $F$, it must be in $C(P)$.

For the converse we observe that the simplices $\conv(\Delta,z)$, $\Delta\in\Sigma_F$, are a triangulation of $\conv(F,z)$. Then the union of $\Hilb(C(P))$, $z'$ and the $\Lpara(\Delta,z)\setminus C(P)$, where $\Delta$ ranges over the $(d-1)$-simplices in $\Sigma_F$, contains a generating set of the monoid $\Lat(C(Q))$. But (b) implies that all the lattice points in $\Lpara(\Delta,z)\setminus C(P)$ are reducible.
\end{proof}

The bound on all quantum jumps over a polytope $P\in\np(d)$ in Theorem \ref{finitestrata} below can be also derived from Theorem \ref{ParaCrit}. However, we present an independent proof which uses a weaker condition than normality, related to the s.c. very ampleness of lattice polytopes.

\begin{definition}\label{VeryAmple}
A lattice polytope $R\subset\RR^d$ is \emph{very ample} if $\Hilb\big(\RR_+(R-v)\big)\subset R-v$  for every vertex $v$ of $R$.
\end{definition}

Normality implies very ampleness but not conversely; very ample polytopes define the normal projective toric varieties and very ample line bundles on them, which also explains the name; a lattice polytope $R\subset\RR^d$ is very ample if and only if the complement $\overline{M}(R)\setminus M(R)$ for the monoids introduced in Section \ref{Facts} is finite. For these and other generalities see  \cite[Sect.~2]{BDGM}.

We have already seen in Theorem \ref{infinitestrata} that the situation drastically changes from dimension $2$ to $3$: there is no uniform limit on the number of empty strata for all $P\in\np(3)$. For a fixed $P$ of any dimension there is however such a bound (even after relaxing the normality condition for $P$).

\begin{theorem}\label{finitestrata}
Let $P\subset\RR^d$ be a (not necessarily very ample) lattice $d$-polytope and let $z$ be a point in $\ZZ^d$ outside $P$. If $\Hilb(\RR_+(P-z))\subset P-z$ then
\begin{equation}
|\height_F(z)|\le 1+ (d-2)\width_F P\label{HBound}
\end{equation}
for every facet $F$ of $P$ that is visible from $z$. In particular, if the polytope $\conv(P,z)$ is very ample and $\Lat(\conv(P,z))=\Lat(P)\cup\{z\}$, then $|\height_F(z)|$ satisfies the bound \eqref{HBound}.
\end{theorem}

\begin{proof}
By applying the parallel translation by $-z$, we can assume $z=0$. Denote $R=\conv(P,0)$. Let $F$ be a facet of $P$, visible from $0$. By Theorem \ref{Encapsulates}(f), the dilated polytope $(d-1)\conv(F,0)$ is normal. Hence, by Theorem \ref{Encapsulates}(e), there exists a lattice point $x\in(d-1)\conv(F,0)$ on lattice height $1$ above the facet $(d-1)F\subset(d-1)\conv(F,0)$.

Because $\Hilb(\RR_+P)\subset P$, the point $x$ is a positive integral linear combination of lattice points of $P$. However, $x$ cannot be the sum of $(d-1)$ or more such points because the $\height_F$-value of the sum will be at least $(d-2)|\height_F(0)|$, whereas $\height_F(x)=(d-2)\height_F(0)-1.$

In particular, $x$ is the sum of at most $(d-2)$ points from $\Lat(P)$. The largest $\height_F$-value of such a sum is $\width_F P+(d-3)(\width_F P+|\height_F(0)|)$, forcing
\begin{align*}
\height_F(x)=(d-2)|\height_F(0)|-1\le\width_F& P+(d-3)(\width_F P+|\height_F(0)|)\\
&\Longrightarrow\quad|\height_F(0)|\le1+(d-2)\width_F P.\qedhere
\end{align*}

\end{proof}

\begin{remark}\label{nonnormal}
One should note that in the special case when $\big(P,\conv(P,z)\big)$ is a jump, Theorem \ref{ParaCrit} contains information beyond the bound in Theorem \ref{finitestrata}: the multiplicity of $F$ also plays an essential role. We have already observed this in Remark \ref{emptyjump}. Furthermore, if $\conv(P,z)$ is normal but $P$ is not, then one can show based on Theorem \ref{ParaCrit} that the bound in Theorem \ref{finitestrata} can be improved to $|\height_F(z)|\le (d-2)\width_F P$.
\end{remark}

We will see below that the bound in Theorem \ref{finitestrata} cannot be improved, not even for quantum jumps of any dimension $d$.

As a consequence the number of lattice points that are candidates for quantum jumps over a polytope $P$ is bounded, and the set of candidates can be efficiently described: the candidates are contained in the set
$$
\Lat\left(P^{-1-(d-2)\width P}\right)\setminus P.
$$
This is the basis of our experiments with quantum jumps that helped us to find maximal elements in $\np(d)$ for $d=4$ and $d=5$.

Our next theorem shows that the bound in Theorem \ref{finitestrata} is sharp even for normal polytopes.

\begin{theorem}\label{sharp}
For every natural number $d\ge2$ and $w\ge1$ there exists a jump $(P,Q)$ of dimension $d$ satisfying the following conditions:
\begin{enumerate}[{\rm(a)}]
\item The vertex of $Q$, not in $P$, is visible from exactly one facet $F\subset P$,
\item $\width_FP=w$,
\item
$\height(P,Q)=(d-2)w+1$.
\end{enumerate}
\end{theorem}

\begin{figure}[hbt]
\begin{center}
\tikzset{facet style/.style={opacity=1.0,very thick,line,join=round}}
\begin{tikzpicture}[x  = {(-0.5cm,-0.5cm)},
                    y  = {(0.9659cm,-0.25882cm)},
                    z  = {(0cm,1cm)},
                    scale = 2]
\draw [->,dashed] (-2.5, 0, 0) -- (2.5,0,0) node at (2.6,0,0) {$x_1$};

\draw [->,dashed] (0, -2, 0) -- (0,2.0,0) node at (0,2.1,0) {$x_2$};

\draw [->,dashed] (0, 0, -1.5) -- (0,0,1.5) node at (0,0,1.6){$x_3$};

\draw[very thin] (0,0,-1) -- (0,0,0) -- (1,0,0) -- cycle;
\draw[very thin] (0,0,-1) -- (0,0,0) -- (0,1,0) -- cycle;
\draw[very thick] (0,0,-1) -- (1,0,0) -- (0,1,0) -- cycle;
\draw[very thin]  (0,0,0) -- (0,1,0) -- (1,1,2) -- cycle;
\draw[very thin]  (0,0,0) -- (1,0,0) -- (1,1,2) -- cycle;
\draw[very thick]  (0,1,0) -- (1,0,0) -- (1,1,2) -- cycle;
\filldraw[fill=white] (0,0,0) circle (1pt);
\filldraw[fill=white] (0,0,-1) circle (1pt);
\filldraw[fill=white] (1,0,0) circle (1pt);
\filldraw[fill=white] (0,1,0) circle (1pt);
\filldraw[fill=white] (1,1,2) circle (1pt) node at (1.3, 1.3,2.4){$z$};
\end{tikzpicture}
\caption{The polytope of Theorem \ref{sharp} for $d=3$, $w=1$}
\end{center}
\end{figure}
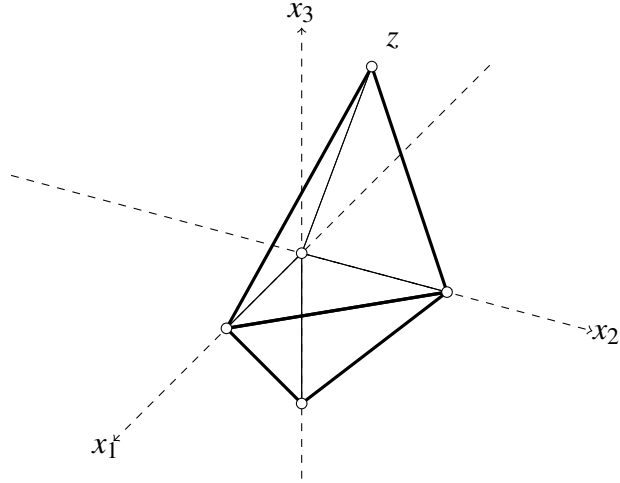

\begin{proof}
There is nothing to show for $d=2$. Therefore we assume $d\ge 3$.

We choose the polytope $P$ to be spanned by the vertices $0$, $\e_1,\dots,\e_{d-1}$ and $-w\e_d$. It is the top element of the unique chain of length $w-1$ in $\np(d)$, starting with the unimodular simplex $\conv(0,\e_1,\ldots,\e_{d-1},-\e_d)$ and finishing with $P$. In particular, $P$ is normal. Over the `horizontal' facet $F$ spanned by $0$ and the $\e_i$, $i\leq d-1$, it has width $w$.

Let
$$
z=(1,\dots,1,(d-2)w+1)=\e_1+\dots+\e_{d-1}+((d-2)w+1)\e_d.
$$
It is easy to check that $z$ is the only additional lattice point in $Q=\conv(P,z)$, that it has height $(d-2)w+1$ over $F$, and that $F$ is the only facet of $P$ that is visible from $F$.

The critical issue is the normality of $Q$. For the application of Theorem \ref{ParaCrit} it is advisable to use homogenized coordinates in $\RR^{d+1}$, as usual indicated by $\vphantom{i}'$.

We claim that the nonzero points in $\Lpara(F,z)$ are given by
$$
y_k=u(\e_1'+\dots+\e_{d-1}') + vz' + t0',
$$
where
\begin{align*}
v&=\frac{k}{w(d-2)+1},\qquad k=1,\dots,w(d-2),\\
u&=1-v,\\
t&= \lceil s \rceil-s,\qquad    s=(d-1)u + v.
\end{align*}
In fact, since $F$ is unimodular and $\vol_d(\conv(F,z))=(d-2)w+1$, there are exactly $w(d-2)$ points in $\Lpara(F,z)$, one on height $k$ above $C(F)$ for each $k=1,\dots,w(d-2)$ (Proposition \ref{ht-points}). Hence the indicated values of $v$ are as above. The sum $u+v$ must be integer and $0\le u<1$, which motivates the value of $u$. The last coordinates of the points $y_k$ are integers and, simultaneously, $0\le t<1$, yielding the indicated values for $t$.

We have
$$
y_k=(1,\dots,1,k,h_k),\qquad k=1,\dots,(d-2)w.
$$

For the difference $y_{(d-2)w+1-k}-z'$ one obtains
\begin{align*}
    (0,...,0,-k,1),\qquad k&=1,...,w,\\
    (0,...,0,-k,2),\qquad k&=w+1,...,2w,\\
    &\vdots\\
    (0,...0,-k,d-2),\qquad k&=(d-3)w+1,...,(d-2)w,
\end{align*}
where the $(d+1)$-st coordinates on the left are computed by the formula
$$
h_k-1=\lceil s\rceil-1=\bigg\lceil\frac{(d-2)(w+k)+1}{(d-2)w+1}\bigg\rceil-1.
$$
All these points lie in $C(P)$, i.e., after dehomogenization with respect to the last coordinate we get points in $P$.
\end{proof}

\begin{remark}
Theorems \ref{Dim3} and \ref{sharp} rely on the exact knowledge of the distribution of the numbers $\height_F(x)$ in the critical areas relative to $\deg x$.

In the proof of Theorem \ref{sharp} there is only a single facet $F$ visible from $z$ and all facets of $\conv(F,z)$ are not only empty, but even unimodular. This follows from the fact that all nonzero elements of $\Lpara(F,z)$ are in the interior. But even under these `optimal' conditions it seems difficult to find a transparent generalization of Theorem \ref{Dim3} to higher dimensions.

It is instructive to compute the heights of the elements of $\Lpara(F,z)$ over the other facets of $\conv(F,z)$ from the proof of Theorem \ref{sharp} for $d=4$, $w=2$. Over $F$ the degree $2$ elements have heights $3$ and $4$, and the degree $3$ elements have heights $1$ and $2$ (and the height $1$ element lets us reach the upper bound). Over the other facets the height distributions are $1,2$ in degree $2$, vs.\ $3,4$ in degree $3$ (three facets)  and  $3,1$ in degree $2$ vs.\ $2,4$ in degree $3$.

This example shows that one cannot predict a priori the distribution of heights over a facet of an empty $4$-simplex, not even if all facets are unimodular.
\end{remark}

\section{Spherical polytopes}\label{Spherical-states}

Throughout this section we fix a natural number $d\ge 2$.

Under certain constrains on the shapes of the normal polytopes we can derive more stringent bounds on the heights of quantum jumps than in Theorem \ref{finitestrata}. More precisely, in this section we show that for asymptotically spherical polytopes the heights of jumps become infinitesimally small compared to the widths. This also naturally leads to interesting number theoretical questions.

\subsection{Asymptotically infinitesimal jumps.}\label{Bounding-jumps2}
Below we will need the following criterion for quantum jumps, which is reminiscent of a dehomogenized version of Theorem \ref{ParaCrit}:

\begin{theorem}\label{norm-criterion}
Let $P\in\np(d)$ with $0\notin P$ and $Q=\conv(P,0)$. Then the following conditions are equivalent:
\begin{enumerate}[{\rm (a)}]
\item
$(P,Q)$ is a jump,
\item
$\Lat\big(kQ\setminus((k-1)Q\cup kP)\big)=\emptyset$ for all $k\in\NN$,
\item
$\Lat\big(kQ\setminus((k-1)Q\cup kP)\big)=\emptyset$ for $k=1,\ldots,d-1$.
\end{enumerate}
\end{theorem}

\begin{proof}


In the following we use the monoids $M(Q)$ and $\overline M(Q)$ and their $k$-th degree components $M(Q)_k$ and $\overline M(Q)_k$, introduced in Section \ref{normal-polytope}.

\smallskip (a)$\implies$(b) 
Consider $x\in \Lat(kQ)$. By normality, $x=\sum_{i=1}^k q_i$ for $q_i\in \Lat(Q)$. If there exists $q_i=0$, we may omit it in the sum, hence $x\in \Lat((k-1)Q)$. Otherwise all $q_i\in \Lat(P)$, hence $x\in \Lat(kP)$.

\smallskip (b)$\implies$(c) is obvious.

\smallskip (c)$\implies$(a) Assume $(P,Q)$ is not a jump. Let $k$ be the smallest natural number for which $M(Q)_k\subsetneq\overline M(Q)_k$ (notation as above). Since $P$ is normal and $k\ge2$, we must have
$$
\Lat\big(kQ\setminus\big((k-1)Q\cup kP\big)\not=\emptyset.
$$
Therefore, (c) implies that the monoids $M(Q)$ and $\overline M(Q)$ coincide up to degrees $d-1$. But then, in view of Lemma \ref{HilbDeg}, the two monoids are equal -- a contradiction.
\end{proof}

The closed $d$-ball in $\RR^d$ with radius $r$ and centered at $z\in\RR^d$ will be denoted by $\B(z,r)$.

\begin{theorem}\label{spherical}
Let $P_i\in\np(d)$, $z_i\in\RR^d$, and $r_i,\epsilon_i$ be positive real numbers, where $i\in\NN$. Assume $\lim_{i\to\infty}r_i=\infty$ and $\B(z_i,r_i-\epsilon_i)\subset P_i\subset\B(z_i,r_i+\epsilon_i)$ for all $i\gg0$.
\begin{enumerate}[{\rm (a)}]
\item
If $\lim_{i\to\infty}\frac{\epsilon_i}{r_i}=0$ then
$$
\lim_{i\to\infty}\max\left(\frac{\height(P_i,Q)}{\width P_i}\ :\ (P_i,Q)\ \text{a jump}\right)=0.
$$
\item
If $\epsilon=\epsilon_1=\epsilon_2=\cdots$ then for all $i\gg0$ and all jumps $(P_i,Q)$ we have
\begin{align*}
\min\big(\|v-x\|\ :\ x\in P_i\big)<47\epsilon+12,
\end{align*}
where $Q=\conv(P_i,v)$.
\end{enumerate}
\end{theorem}

\begin{remark}\label{spherical-ellipsoidal}
Theorem \ref{spherical}(a) and its proof straightforwardly extend to the more general families of polytopes when instead of spheres one uses ellipsoids -- with fixed eccentricities in a family. We present the argument only in the spherical case in order to avoid cumbersome notation. Ellipsoids will appear explicitly in the next subsection in a more number theoretical context. Also, the proof below uses the following weaker condition than normality: the lattice points in the 2nd multiples of the polytopes in question are the sums of pairs of lattice points in the original polytopes.
\end{remark}

\begin{remark}\label{obtuse}
The strong metric bound in Theorem \ref{spherical}(b) does not necessarily translate into a strong bound for the corresponding heights. In fact, as $i$ gets larger, some facets of $P_i$ get increasingly \emph{sloped}, i.e., the ratio of the lattice and metric widths with respect to facets of $P_i$ can be made arbitrarily small as $i\to\infty$. In fact, the normal unit vectors to the facets of $P_i$ define increasingly dense subsets of the unit sphere $S^{d-1}$ as $i\to\infty$. In particular, any neighborhood of any integer nonzero point $w\in\ZZ^d$ meets a hyperplane of the form $\Aff(F)-v$ for some $F\in\FF(P_i)$ and $v\in\vertex(F)$ when $i\gg0$, depending on the neighborhood, so that $w$ is not in the hyperplane. Actually, the same argument shows that, when $i\to\infty$, the absolute majority of the facets of $P_i$ get increasingly sloped.
\end{remark}

\begin{proof}[Proof of Theorem \ref{spherical}] (a) First we observe that the claim follows from the following equality
\begin{equation}\label{almost-radius}
\lim_{i\to\infty}\max_v\left(\frac{\|z_i-v\|}{r_i}\ :\ \big(P_i,\conv(P_i,v)\big)\ \text{a jump}\right)=1,\qquad i\in\NN.
\end{equation}

In fact, the equality (\ref{almost-radius}) is equivalent to the claim that the ratios
$$
\frac{\|z_i-v\|-r_i}{r_i},
$$
where $\big(P_i,\conv(P_i,v)\big)$ is a jump, can be made arbitrarily close to $0$ by choosing $i$ sufficiently large. For an index $i$ and a jump $\big(P_i,\conv(P_i,v)\big)$ pick a facet $F_i\subset P_i$, visible from $v$ and such that
$$
\height\big(P_i,\conv(P_i,v)\big)=\height_{P_i}(v)=-\height_{F_i}(v).
$$
For all $i\gg0$ and all jumps $\big(P_i,\conv(P_i,v)\big)$ we have the inequalities
\begin{align*}
\frac{\height\big(P_i,\conv(P_i,v)\big)}{\width P_i}&\le\frac{\height_{P_i}(v)}{\width_{F_i}P_i}\le\\
&\frac{\inf_{\Aff(F_i)}\|\Aff(F_i)-v\|}{2(r_i-\epsilon_i)}\le\left|\frac{\|z_i-v\|-r_i+\epsilon_i}{2(r_i-\epsilon_i)}\right|.
\end{align*}
So Theorem (\ref{spherical})(a) follows from (\ref{almost-radius}).

Next we prove (\ref{almost-radius}). Assume to the contrary that the considered limit is not $1$. This means that infinitely many of the considered ratios exceed $1$ by some real number $\theta>0$. After picking the corresponding subsequence and re-indexing, we can assume that there are jumps $Q_i=\big(P_i,\conv(P_i,v_i)\big)$ such that
\begin{equation}\label{separated-from-0}
\theta_i:=\frac{\|z_i-v_i\|-r_i}{r_i}\ge\theta,\qquad i\in\NN.
\end{equation}

Applying the parallel translations by the vectors $-v_i$, we can further assume that $v_i=0$ for all $i$. By Theorem \ref{norm-criterion}, we have
\begin{equation}\label{one-two-multiples}
\Lat\big(2Q_i\setminus(Q_i\cup 2P_i)\big)=\emptyset,\qquad i\in\NN.
\end{equation}

As $i\to\infty$, the subsets
\begin{align*}
T_i:=\conv\big(\B\big((2&+2\theta_i)\e_1,2\big)\cup\{0\}\big)\setminus\\
&\big(\conv\big(\B\big((1+\theta_i)\e_1,1\big)\cup\{0\}\big)\cup\B((2+2\theta_i)\e_1,2)\big)\subset\RR^d
\end{align*}
become approximately congruent with increased precision to the rescaled subsets
$$
\frac1{r_i}\big(2Q_i\setminus(Q_i\cup 2P_i)\big)\subset\RR^d.
$$
The following can be said on the geometry of the closure $\bar T_i\subset\RR^d$ of $T_i$ in the Euclidean topology:
\begin{enumerate}[{\rm (i)}]
\item
$\bar T_i$ is homeomorphic to a $d$-torus;
\item
$\bar T_i$ is invariant under rotation of $\RR^d$ about the axis
$\RR \e_1$.
\end{enumerate}

These properties, together with the inequalities (\ref{separated-from-0}), imply the existence of a real number $\rho>0$ such that for every index $i$ the set $\bar T_i$ contains a ball $\B_i'$ of radius $\rho$. The easiest way to see this is by induction over $d$: the case $d=2$ is obvious and every such ball $\B_{d'}$ of radius $\rho'$ in dimension $d'<d$ gives rise by revolution about the axis $\RR_+\e_1$ to a torus of dimension $d'+1$, which in its turn contains a $(d'+1)$-ball of radius $\rho$ only depends on the radius of $\B_{d'}$; see Figure \ref{ball-cone}. (It is an exercise to show that, actually, we can take $\rho'=\rho$.)

\begin{figure}[h!]
\vspace{.15in}
\includegraphics[trim = 0mm 0.1in 0mm 0.1in, clip, scale=.5]{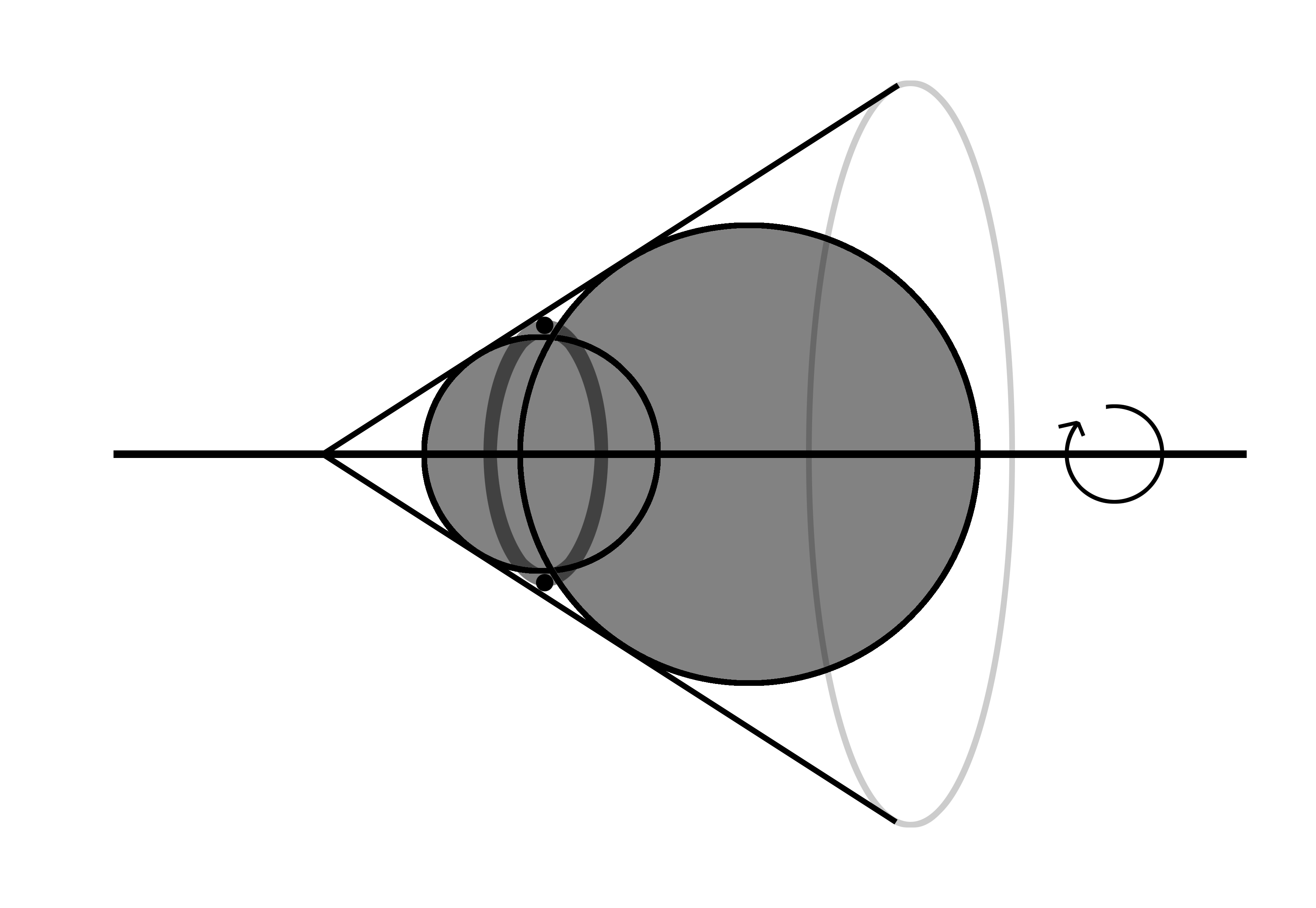}
\caption{Generating a $d$-torus in $\bar T_i$ by revolving a $(d-1)$-ball}\label{ball-cone}
\end{figure}

We see that, for any real number $0<\kappa<1$, the subsets
$$
2Q_i\setminus(Q_i\cup 2P_i)\subset\RR^d
$$
contain balls of radius $\kappa\rho r_i$ whenever $i\gg0$. But this contradicts (\ref{one-two-multiples}) because $\kappa\rho r_i\to\infty$ as $i\to\infty$.

\MEDSCIP\NOINDENT (b) Pick a sequence of jumps $(P_i,Q_i)$, $i\in\NN$. Without loss of generality we can assume $\Lat(Q_i)\setminus P_i=\{0\}$.
By Theorem \ref{norm-criterion}, we have the same equality (\ref{one-two-multiples}) as in the proof of the part (a).

For simplicity of notation, put $\B_i=\B(z_i,r_i)$. Consider the subset
$$
R_i=\conv(2\B_i\cup \{0\})\setminus\big(\conv(\B_i\cup\{0\})\cup 2\B_i\big)\subset\RR^d
$$
and consider the closure $\bar R_i\subset\RR^d$ of $R_i$ in the Euclidean topology. The set $\bar R_i$ is homeomorphic to a $d$-torus, invariant under rotation of $\RR^d$ about the axis $\RR z_i$.

Let $C_i\subset\RR^d$ be the cone $\RR_+\B_i$. Then the boundary $\partial C_i$ is tangent to the ball $\B_i$. Let $\delta_i$ denote the distance from $0$ to any point $B_i\cap\partial C_i$.

First we observe that the part (a) implies
\begin{equation}\label{delta-zero}
\lim_{i\to\infty}\frac{\delta_i}{r_i}=0.
\end{equation}

For every index $i$, the set of farthest points of $\bar R_i$ from $\partial C_i$, which can be connected to $\partial C_i$ by segments perpendicular to $\partial C_i$ and entirely inside $\bar R_i$, form a circle $S_i$ with center in the line $\RR_+ z_i$. Assume the distance from $S_i$ to $\partial C_i$ is $h_i$. Then
\begin{equation}\label{length-estimate}
\delta_i=\sqrt{4r_i^2-(2r_i-h_i)^2}+\sqrt{r_i^2-(r_i-h_i)^2},
\end{equation}
as follows from Figure \ref{metric}, representing a section by any 2-dimensional plane in $\RR^d$ through $0$ and $z_i$.
\begin{figure}[h!]
\begin{center}
\footnotesize
\tikzset{facet style/.style={opacity=1.0,very thick,line,join=round}}
\begin{tikzpicture}
\draw[fill=black] (0,0) circle (2pt);
\draw[thick] (0,0) -- (8,0) node at (-0.2,0){0};
\draw[thick] (0,0) -- (0.8,9.144);
\draw[thick] (2,1.8326) circle (1.8326);
\draw[fill=black] (2,1.8326) circle (2pt) node at (2,2.13){$z_i$};
\draw[fill=black] (2,0) circle (1pt);
\draw[thick] (2,1.8326) -- (2,-0.5) node at (1.8,1.4){$r_i$};
\draw[thick] (4,3.6652) circle (3.6652);
\draw[thick] (4,3.6652) -- (4,-0.5) node at (4.4,2.3){$2r_i$};
\draw[thick] (0,0) -- (0,-0.5);  
\draw[fill=black] (4,3.6652) circle (2pt) node at (4,4.0){$2z_i$};
\draw[fill=black] (4,0) circle (1pt);
\draw[fill=black] (2.822,0.195) circle  (2pt) node at (2.822,0.7){$x_i$} node at (2.822,-0.2){$y_i$};
\draw[very thin] (1.6,0.195) --  (4.4,0.195) ;
\draw[fill=black] (2.822,0) circle (2pt);
\draw[thick] (2.822,0.195) -- (2.822,0);
\draw[thick] (2,1.8326) -- (2.822,0.195) -- (4,3.6652);
\draw[very thin, <->] (0,-0.5) -- (2,-0.5) node at (1,-0.8){$\delta_i$};
\draw[very thin, <->] (2,-0.5) -- (4,-0.5) node at (3,-0.8){$\delta_i$};
\draw[very thin, ->] (4.3,0.695) -- (4.3,0.195) node at (6,0.2){$h_i$};
\draw[very thin, ->] (4.3,-0.5) -- (4.3,0);
\end{tikzpicture}
\end{center}
\vspace{.15in}
\caption{Planar cross section}\label{metric}
\end{figure}
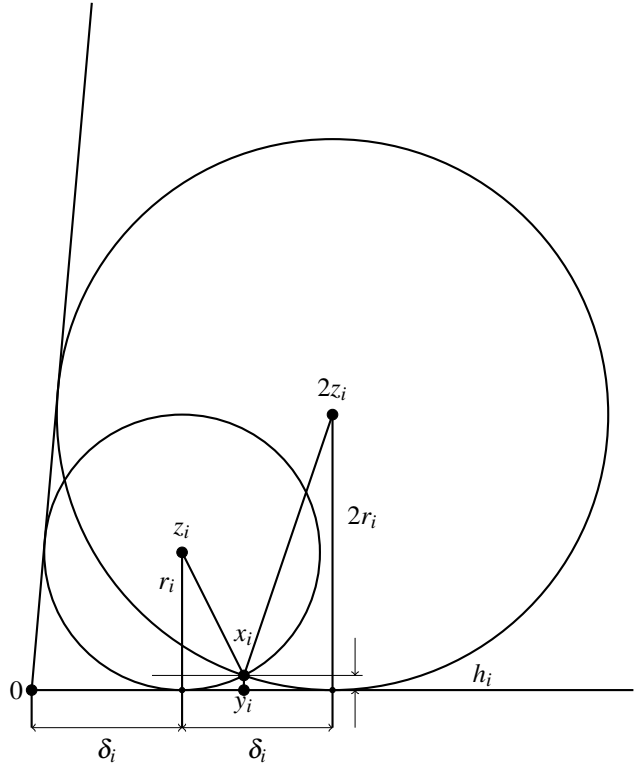

For every index $i$, pick a point $x_i\in S_i$ and let $[x_i,y_i]$ be the segment inside $\bar R_i$, perpendicular to $\partial C_i$ and with $y_i\in\partial C_i$. (In particular, $\|x_i-y_i\|=h_i$). Because $r_i\to\infty$ and (\ref{delta-zero}), the boundary $\partial\bar R_i$ close to the point $x_i$ and $y_i$ becomes increasingly close to the two $(d-1)$-dimensional affine hyperplanes through $x_i$ and $y_i$, perpendicular to $[x_i,y_i]$. In fact, $r_i\to\infty$ implies that the cones $C_i$ become increasingly obtuse, flattening $\partial\bar R_i$ close to $y_i$ as $i\to\infty$, and (\ref{delta-zero}) implies that the balls $B_i$ and $2B_i$ have increasingly large radii but they stay close to each other relative to the radii, flattening $\partial\bar R_i$ close to $x_i$ as $i\to\infty$.  Consequently, as $i\to\infty$, the tori $\bar R_i$ contain right cylinders of arbitrarily large radius around the axes $x_i+\RR(y_i-x_i)$ with heights arbitrarily close to $h_i$.  Fix a system of such cylinders $\Pi_i\subset\bar R_i$. We can assume that the heights of the $\Pi_i$ are more than $\frac12h_i$.

\MEDSCIP We have
\begin{equation}\label{new-containment}
\begin{aligned}
2Q_i\setminus(Q_i\cup 2P_i)\supset\conv&\big(\B\big(2z_i,2r_i-2\epsilon\big)\cup\{0\}\big)\setminus\\
&\big(\conv\big(\B_i(z_i,r_i+\epsilon)\cup\{0\}\big)\cup2B_i(2z_i,2z_i+2\epsilon)\big).
\end{aligned}
\end{equation}

Using the same flatness of $\partial\bar R_i$ close to $x_i$ and $y_i$, one concludes that the intersection of $\Pi_i$ with the right side of (\ref{new-containment}) contains a coaxial right sub-cylinder of height
$>(\text{the height of}\ \Pi_i-4\epsilon)>\frac12h_i-4\epsilon$ and the same radius as $\Pi_i$, provided $i\gg0$.

If $\frac12h_i-4\epsilon>1$ for infinitely many indices $i$ then the mentioned cylinders contain lattice points for $i\gg0$, contradicting (\ref{one-two-multiples}) in view of the containments (\ref{new-containment}).

Since the functions $f_i(x)=4r^2_i-(2r_i-x)^2$ and $g_i(x)=r_i^^2-(r_i-x)^2$ are increasing over the segment $[0,r_i]$, $\lim_{i\to\infty}r_i=\infty$, and $\frac12h_i-4\epsilon\le1$ for $i\gg0$, the equalities (\ref{length-estimate}) imply
\begin{align*}
\delta_i\le\sqrt{4r_i^2-(2r_i-2-8\epsilon)^2}+\sqrt{r_i^2-(r_i-2-8\epsilon)^2}<\sqrt{r_i}(2+\sqrt2)\sqrt{2+8\epsilon},
\end{align*}
provided $i\gg0$.

Finally, for all $i\gg0$ we have
\begin{align*}
\min\big(\|x\|\ &:\ x\in P_i\big)\le\|z_i\|-r_i+\epsilon=\sqrt{r_i^2+\delta_i^2}-r_i+\epsilon=\\
&\frac{\delta_i^2}{\sqrt{r_i^2+\delta_i^2}+r_i}+\epsilon<\frac{r_i(2+\sqrt2)^2(2+8\epsilon)}{2r_i}+\epsilon<47\epsilon+12.\qedhere
\end{align*}
\end{proof}

A more careful choice in the cylinders inside $\bar R_i$ in the argument above leads to a better estimate in Theorem \ref{spherical}(b), but in view of Remark \ref{obtuse} such an improvement is not worth pursuing.

\subsection{Convex hulls of all lattice points in spheres}\label{Number-theory}
There is a ubiquity of sequences $\{P_i\}_{i\in\NN}$, satisfying the stronger condition in Theorem \ref{spherical}(b), which one could call \emph{rapidly spherical families}. Here is one recipe for deriving such a sequence. Choose any divergent series of real numbers $0<r'_1<r'_2<\ldots$ and put $P_i'=\conv(\Lat(\B(r'_i,0))$. Because every unit $d$-cube in $\RR^d$ contains a lattice point and every $d$-ball $\B\subset\RR^d$ of radius $\frac{\sqrt d}2$ contains a unit cube, we have
$$
\vertex(P'_i)\subset\B(0,r'_i)\setminus\B(0,r'_i-\sqrt d/2),\quad i\in\NN.
$$
Fix an arbitrary real number $\theta>0$. The inclusions above imply
$$
\B(0,r'_i-(1+\theta)\sqrt d/2)\subset P'_i\subset\B(0,r'_i),\quad i\gg0.
$$
By Theorem \ref{Encapsulates}(f), the polytopes $P_i=(d-1)P_i'$ are normal for all $i$ and we also have
$$
\B(0,r_i-\epsilon)\subset P_i\subset\B(0,r_i+\epsilon),\quad i\gg0,
$$
where
$$
r_i=(d-1)\big(r'_i-(1+\theta)\sqrt d/4\big)\quad\text{and}\quad\epsilon=(d-1)(1+\theta)\sqrt d/4.
$$

\smallskip Similar examples can be derived when instead of balls one uses ellipsoids of fixed eccentricities per a family, not necessarily centered at $0$.

Dilated lattice polytopes usually have non-empty first lattice strata around them. In particular the proposed recipe for deriving  rapidly ellipsoidal families are unlikely to represent maximal elements in $\np(d)$. This observation motivates the interest in studying the normality of the convex hulls of \emph{all} lattice points in spheres or, more generally, ellipsoids. We have the following partial result.

\begin{theorem}\label{Ball-normal}
Let $l_1,\ldots,l_d$ be linearly independent real linear $d$-forms and $(z_1,\dots,z_d)\in \RR^d$. Consider the ellipsoid
$$
E=\left\{\xi=(\xi_1,\ldots,\xi_d)\ :\ (l_1(\xi)-z_1)^2+\cdots+(l_d(\xi)-z_d)^2\le 1\right\}\subset\RR^d,
$$
and the polytope $P=\conv(\Lat(E))$.
\begin{enumerate}[{\rm(a)}]
\item For any integer $k\geq 2$ and any point $y\in kP$ there exists a point $w\in \Lat(P)$ such that $y-w\in (k-1)E$.
\item For any $y\in \Lat(2P)$ there exist $w_1,w_2\in \Lat(P)$ such that $y=w_1+w_2$.
\item If $d=3$ then $P$ is a normal polytope.
\end{enumerate}
\end{theorem}

\begin{proof} (a) For simplicity of notation, put $\sum=\sum_{i=1}^d$. Consider the (potentially $0$) linear form
$$
l(\xi)=\sum l_i(\xi)(l_i(y)-kz_i).
$$
As $y/k\in P$ there must exist a vertex $w\in P$ such that $l(w)\geq l(y/k)$, i.e.,
$$
\sum l_i(w)(l_i(y)-kz_i)\geq \sum \frac{l_i(y)}{k}(l_i(y)-kz_i).
$$
This is equivalent to
$$
\sum (l_i(w)-z_i)(l_i(y)-kz_i)\geq \sum\left(\frac{l_i(y)}{k}-z_i\right)(l_i(y)-kz_i)
$$
and, therefore, to
\begin{equation}\label{eq:nier}
\sum 2(l_i(y)-kz_i)(l_i(w))-z_i)\geq \frac{2}{k} \sum\big(l_i(y)-kz_i\big)^2.
\end{equation}

We have
\begin{align*}
\sum&\big(l_i(y-w)-(k-1)z_i\big)^2=\sum\big((l_i(y)-kz_i)-(l_i(w)-z_i)\big)^2=\\
&\sum\big(\big(l_i(y)-kz_i\big)^2-2\big(l_i(y)-kz_i\big)\big(l_i(w)-z_i\big)+\big(l_i(w)-z_i\big)^2\big),
\end{align*}
which, in view of (\ref{eq:nier}), implies
$$
\sum\big(l_i(y-w)-(k-1)z_i\big)^2\leq \frac{k-2}{k}\sum\big(l_i(y)-kz_i\big)^2+\sum\big(l_i(w)-z_i\big)^2.
$$
As $y\in kE$ and $w\in E$ we obtain:
$$
\sum\big(l_i(y-w)-(k-1)z_i\big)^2\leq \frac{k-2}{k}\cdot k^2+1=(k-1)^2,
$$
i.e., $y-w\in (k-1)E$.

\MEDSCIP\NOINDENT(b) We choose $w$ as in (a). Then $y-w\in\Lat(E)=\Lat(P)$.

\MEDSCIP\NOINDENT(c) follows from (b) because of Lemma \ref{HilbDeg}.
\end{proof}


\begin{remark}
(a) We have tested several dozens of polytopes defined by ellipsoids with axes parallel to the coordinate axes in dimensions four and five, all of which turned out to be normal.

(b) The standard three dimensional balls $\B(0,r)$, $r=1,2,\dots,21$, define nonmaximal polytopes: all of them have height $1$ jumps. The maximal height of jumps over them varies in an irregular manner: the smallest is $2$ for $r=2$, the largest is $11$ for $r=13$, and for $r=21$ it is $9$. Despite of its irregular behavior, the maximal height of jumps seems to grow slowly with $r$.

\end{remark}

\section{Explicit maximal polytopes}\label{Maximal states}
We have found maximal polytopes of dimension $4$ and $5$. This leaves little doubt that there exist maximal polytopes of any dimension $\ge 4$, but dimension $3$ remains open. The experiments described in this section are based on a computer program written in C++ that makes heavy use of the library interface of Normaliz \cite{Nmz}.

We want to emphasize that the experiments described below have not only produced maximal polytopes, but have also motivated several central results of the preceding sections.

\subsection{The extension approach}
The basic search strategy for finding maximal elements by successive extension is very simple:
\begin{enumerate}
\item Choose a normal start polytope $P$.
\item If $\#\Lat(P)$ exceeds a preset bound, go to (1).
\item Find a jump $Q$ over $P$.
\item If none exists, stop and save the maximal polytope $P$.
\item Replace $P$ by $Q$ and go to (2).
\end{enumerate}

In addition to special constructions, like the cross-polytopes, we have implemented two methods for finding a start polytope:

\begin{enumerate}
\item[(U)] Take the unimodular $d$-simplex and extend it by a random number of random height $1$ jumps. The polytope thus reached is considered the start polytope.

\item[(S)] We start from a lattice parallelotope and `shrink' it successively by removing a vertex and taking the convex hull of the remaining vertices until no vertex can be removed without losing normality or the full dimension. The reached polytope serves as the starting point for subsequent extensions.
\end{enumerate}

When we say `random', we mean the choice of a random integer or vector within a certain range that can be modified via parameters of the search program.

At first glance, the shrinking technique (S) which was seems paradoxical: we shrink a parallelotope and then extend the shrunk polytope in order to reach a maximal one. However, (S) has proved very successful. Also (U) has led to maximal polytopes.

We can apply various strategies for finding quantum jumps over $P$. There are two major variants:
\begin{enumerate}
\item[(\textbf1)]Choose a height $1$ jump at random, provided such exists.
\item[(P)] Choose a jump which maximizes a certain parameter, meant to lead to some sort of irregular normal polytopes.
\end{enumerate}

If (\textbf1) is applied, one needs to compute only the points in the first stratum around the given polytope, and this is usually quite fast. Moreover, there is no need to test if the candidates are really jumps. For (P) we compute all candidate points according to Theorem \ref{finitestrata} in dimension $3$, but use a lower bound in dimensions $\ge 4$ for the search phase, applying the full bound in the verification phase only.

The polytopes containing the candidates are highly rational. Nevertheless their lattice points can be computed very fast via the approximation algorithm of Normaliz.

It might seem most promising to always apply strategy (P), for example with the volume of the jump. But in pure form it has two drawbacks: (i) it tends to create successive jumps along straight lines that are not limited, and (ii) it is rather time consuming to test all candidate points in decreasing order of volume.

The following mixed strategy for step (3) of the basic algorithm has led us to the maximal polytopes $P_4$ and $P_5$ described below (and many others):
\begin{enumerate}
\item[(3a)] Extend $P$ according to (\textbf1) if a height $1$ jump exists.
\item[(3b)] Otherwise apply (P).
\end{enumerate}
The two parameters for (P) that have proved successful are
\begin{enumerate}
\item[(V)] the volume of the jump, see Remark \ref{measures};
\item[(A)] the average multiplicity (or normalized $(d-1)$-volume) of the facets of $Q$.
\end{enumerate}
In fact, the larger the multiplicity of a facet $F$, the more lattice points of low height over $F$ in $P,\dots,(d-2)P$ are necessary to guarantee normality of the extension; see Theorem \ref{ParaCrit}. It is not surprising that the facet multiplicities of the maximal polytopes are quite large; see Table \ref{Widths}.

\subsection{The random generation approach} In this approach we
\begin{enumerate}
\item[(1)] Choose a normal polytope at random and
\item[(2)] check it for maximality.
\end{enumerate}
Creating a normal polytope by randomly choosing vertices becomes more and more difficult with growing dimension and number of vertices. According to our experience it works very well in dimension $4$ if we limit ourselves to simplices.

The main advantage of this brute force approach is the enormous number of candidate polytopes that can be scanned if one gives up the idea of successive extension, and one can say that even in mathematics mass production may beat sophistication.

The random generation approach has produced the simplex $P_4'$ below, many others in $\np(4)$ and two in $\np(5)$, of which one has only $21$ lattice points.

The frequency of hitting maximal elements of $\np(d)$ in our computations so far has been more or less the same for the two methods.

\subsection{Some maximal polytopes}

Table \ref{Vertices} contains the vertices of some maximal polytopes.
\begin{table}[hbt]
\begin{align*}
P_4:\ &(0,0,0,0)\quad&     P_5:\ &( 4 -13,-2,-1,1)\quad& P_4':\ &  (0,3,2,0)\\
&(3,0,2,0)\quad &    &( 4,12,13,4,-2)\quad &                    &  (1,1,3,2)\\
&(-2,-3,3,-1)\quad&  &(-2,0,-8,-2,1)\quad &                     &  (2,3,0,4)\\
&(10,3,-3,-1)\quad&  &( 0,-2,0,0,0)\quad &                      &  (4,0,0,2)\\
&(0,-3,1,-2)\quad&   &(27 -26 -15,-6,3)\quad &                  &  (4,4,4,2)\\
&(2,-2,0,-2)\quad&   &(10,-1 -11,-4,1)\quad &    &\\
&(-9,4,10,4)\quad&   &(10 -13,-2,-1,1)
\end{align*}
\caption{Vertices of maximal polytopes in dimensions $4$ and $5$}
\label{Vertices}
\end{table}
The numbers of lattice points are $41$ in $P_4$, $42$ in $P_5$ and only $22$ in $P_4'$. Note that $P_4'$ is a simplex with $22$ lattice points. These numbers are small in relation to the widths of the polytopes over their facets that we have listed in Table \ref{Widths} together with the multiplicities of the facets. Although we have no analogue of Theorem \ref{Dim3} in higher dimensions, one can expect that a maximal polytope has few lattice points relative to its facet widths and multiplicities.

By now, more than $40$ maximal polytopes have emerged in dimension $4$ and $6$ in dimension $5$. Despite of millions  of attempts with varying strategies, our search has been futile in dimension $3$.

For the three maximal polytopes the second lattice stratum is nonempty. In other words there exist height $2$ points over $P_4$, $P_5$ and $P_5'$.  There also exist maximal polytopes whose first two strata are empty.

\begin{table}[hbt]
\begin{tabular}{r|r|r|r|r|r|r|r|r|r|r|r|r}
$P_4$ &width& 29&180&66&8&20&116&40&91&32&80&160\\\hline
& mult&4&1&4&10&4&1&2&4&10&4&2
\end{tabular}
\vspace*{14pt}

\begin{tabular}{r|r|r|r|r|r|r|r|r|r|r|r}
$P_5$&width&27&105&24&24&105&105&48&105&27&105\\\hline
&mult&18&9&18&18&9&9&9&9&18&9
\end{tabular}
\vspace*{14pt}

\begin{tabular}{r|r|r|r|r|r|r}
$P_4'$&width&24&48&48&48&48\\\hline
&mult&8&4&4&4&4
\end{tabular}
\vspace*{14pt}

\caption{Widths and facet multiplicities of maximal polytopes}
\label{Widths}
\end{table}

We add a few data of the computations for $P_4$ and $P_5$. The number of lattice points satisfying  the height bound of Theorem \ref{finitestrata} are $196,697$ for $P_4$ and $13,525,003$ for $P_5$. The computation of these candidate points takes $<2$ sec for $P_4$ and $<7$ min for $P_5$.

In order to verify that a candidate point is not a quantum jump, we first check whether $\#\Lat(\conv(P,z))=\#\Lat(P)+1$. Only few candidates survive, namely $84$ for $P_4$ and $980$ for $P_5$. For these we compute the Hilbert bases of the extended polytope and look for Hilbert basis elements of degree $>1$. The computation times for the verifications are $<2$ min for $P_4$ and $<2.5$ h for $P_5$. The verifications are documented in log files that list every candidate together with a `witness', namely an extra element of $\Hilb(\conv(P,z))$. (The computations were done on a system with an Intel Xeon CPU E5-2660 0 at 2.20 GHz in strictly serial mode, the data available on request from the authors.)

\begin{remark}
(a) We have checked that the $4$-dimensional maximal polytopes $P_4$ and $P_4'$ remain maximal if we consider very ample polytopes instead of normal ones.

(b) It is obvious that our findings rely crucially on the correctness of Normaliz. In order to enhance our confidence we have verified the maximality of $P_4$ with the dual algorithm of Normaliz. It takes considerable more time than the primal algorithm.

(c) $P_5$ and $P_4'$ have nontrivial symmetries: their automorphism groups are isomorphic to $\ZZ_2\times\ZZ_2$ and $\ZZ_2$, respectively.
\end{remark}

The techniques employed in our experiments, apart of random generation, follow descending and ascending chains in $\np(d)$. They can hardly find polytopes that are simultaneously minimal and maximal.

We end with the following question to which we were naturally led in this paper and which seems very difficult at present.

\begin{question}\label{Question}
\begin{enumerate}[{\rm(a)}]
\item
Does $\np(3)$ have maximal elements?  Does it have nontrivial minimal elements?
\item
Is the convex hull of all lattice points in every ellipsoid normal?
\item
Does there exist a normal polytope that is both a minimal and maximal element of the poset $\np(d)$?
\end{enumerate}
\end{question}

\bibliography{references_polytopes}
\bibliographystyle{plain}

\end{document}